\documentclass[11pt]{article}
\usepackage[normalem]{ulem}

\usepackage{mathpazo}   
\usepackage{tikz}

\usepackage{amssymb}
\DeclareUnicodeCharacter{1D54F}{$\mathbb{X}$}

\usepackage{nameref}
\usepackage{empheq}
\usepackage{comment}
\usepackage[shortlabels,inline]{enumitem}
 \setlist[enumerate]{itemsep=0pt,parsep=0pt,topsep=0pt,partopsep=0pt}
\usepackage[colorlinks=true,
linkcolor=refkey,
urlcolor=lblue,
citecolor=red]{hyperref}
\usepackage[capitalize,nameinlink]{cleveref}
\usepackage{aliascnt} 
\usepackage[doc,wmm,hhb]{optional}
\usepackage{xcolor}

\usepackage{float}
\usepackage{soul}
\usepackage{graphicx}
\usepackage{subfig}

\usepackage{stmaryrd}
\usepackage{amssymb}
\usepackage[margin=1in]{geometry}

\definecolor{labelkey}{rgb}{0,0.08,0.45}
\definecolor{refkey}{rgb}{0,0.6,0.0}
\definecolor{Brown}{rgb}{0.45,0.0,0.05}
\definecolor{lime}{rgb}{0.00,0.8,0.0}
\definecolor{lblue}{rgb}{0.5,0.5,0.99}
\definecolor{OliveGreen}{rgb}{0,0.6,0}
\definecolor{tyrianpurple}{rgb}{0.4, 0.01, 0.24}

\newenvironment{namedproof}[1]{%
  \par\noindent\textbf{#1.}\ \ignorespaces
}{%
  \hfill\ensuremath{\blacksquare}\par
}

\colorlet{hlcyan}{cyan!30}

\makeatletter
\def\namedlabel#1#2{\begingroup
	\def\@currentlabel{#2}%
	\label{#1}\endgroup
}
\makeatother

\newcommand{\seppthree}{\setlength{\itemsep}{-3pt}}

\newcommand*{\tran}{^{\mkern-1.5mu\mathsf{T}}}

\providecommand{\siff}{\Leftrightarrow}
\newcommand{\weakly}{\ensuremath{\:{\rightharpoonup}\:}}

\newcommand{\knn}{\ensuremath{{k\in{\mathbb N}}}}

\newcommand{\menge}[2]{\big\{{#1}~\big |~{#2}\big\}}

\newcommand{\fenv}[1]%
{\ensuremath{\,\overrightarrow{\operatorname{env}}_{#1}}}
\newcommand{\benv}[1]%
{\ensuremath{\,\overleftarrow{\operatorname{env}}_{#1}}}

\newcommand{\scal}[2]{\left\langle{#1},{#2}  \right\rangle}

\newcommand{\RR}{\ensuremath{\mathbb R}}

\newcommand{\NN}{\ensuremath{\mathbb N}}

\DeclareMathOperator*{\argmin}{argmin}
\newcommand{\prox}{\ensuremath{\operatorname{Prox}}}

\newcommand{\parl}{\ensuremath{\operatorname{par}}}

\newcommand{\ran}{\ensuremath{{\operatorname{ran}}\,}}
\newcommand{\zer}{\ensuremath{\operatorname{zer}}}

\newcommand{\Id}{\ensuremath{\operatorname{Id}}}

\newcommand{\dist}[2]{\operatorname{dist}\!\left(#1,#2\right)}
\newcommand{\distsq}[2]{\operatorname{dist}^2\!\left(#1,#2\right)}

\newcommand{\minimize}[2]{\ensuremath{\underset{\substack{{#1}}}{\mathrm{minimize}}\;\;#2 }}

\newcommand{\sperp}{{\scriptscriptstyle\perp}}

{\begin{list}{}{%
			\settowidth{\labelwidth}{\textrm{#1~}}%
			\setlength{\leftmargin}{\labelwidth+\labelsep}}}
{\end{list}}

\crefname{equation}{}{equations}
\crefname{chapter}{Appendix}{chapters}
\crefname{item}{}{items}
\crefname{enumi}{}{}
\crefname{appsec}{Appendix}{Appendices}

\newtheorem{theorem}{Theorem}[section]
\newaliascnt{lemma}{theorem}
\newtheorem{lemma}[lemma]{Lemma}
\aliascntresetthe{lemma}
\crefname{lemma}{Lemma}{Lemmas}
\Crefname{lemma}{Lemma}{Lemmas}

\newaliascnt{lem}{theorem}

\aliascntresetthe{lem}
\crefname{lem}{Lemma}{Lemmas}
\Crefname{lem}{Lemma}{Lemmas}

\newaliascnt{corollary}{theorem}
\newtheorem{corollary}[corollary]{Corollary}
\aliascntresetthe{corollary}
\crefname{corollary}{Corollary}{Corollaries}
\Crefname{corollary}{Corollary}{Corollaries}

\newaliascnt{cor}{theorem}

\aliascntresetthe{cor}
\crefname{cor}{Corollary}{Corollaries}
\Crefname{cor}{Corollary}{Corollaries}

\newaliascnt{proposition}{theorem}
\newtheorem{proposition}[proposition]{Proposition}
\aliascntresetthe{proposition}
\crefname{proposition}{Proposition}{Propositions}
\Crefname{proposition}{Proposition}{Propositions}

\newaliascnt{prop}{theorem}

\aliascntresetthe{prop}
\crefname{prop}{Proposition}{Propositions}
\Crefname{prop}{Proposition}{Propositions}

\newaliascnt{definition}{theorem}

\aliascntresetthe{definition}
\crefname{definition}{Definition}{Definitions}
\Crefname{definition}{Definition}{Definitions}

\newaliascnt{defn}{theorem}

\aliascntresetthe{defn}
\crefname{defn}{Definition}{Definitions}
\Crefname{defn}{Definition}{Definitions}

\newaliascnt{thm}{theorem}

\aliascntresetthe{thm}
\crefname{thm}{Theorem}{Theorems}
\Crefname{thm}{Theorem}{Theorems}

\newaliascnt{example}{theorem}
\newtheorem{example}[example]{Example}
\aliascntresetthe{example}
\crefname{example}{Example}{Examples}
\Crefname{example}{Example}{Examples}

\newaliascnt{ex}{theorem}
\newtheorem{ex}[ex]{Example}
\aliascntresetthe{ex}
\crefname{ex}{Example}{Examples}
\Crefname{ex}{Example}{Examples}

\newaliascnt{fact}{theorem}
\newtheorem{fact}[fact]{Fact}
\aliascntresetthe{fact}
\crefname{fact}{Fact}{Facts}
\Crefname{fact}{Fact}{Facts}

\newaliascnt{remark}{theorem}
\newtheorem{remark}[remark]{Remark}
\aliascntresetthe{remark}
\crefname{remark}{Remark}{Remarks}
\Crefname{remark}{Remark}{Remarks}

\newaliascnt{rem}{theorem}

\aliascntresetthe{rem}
\crefname{rem}{Remark}{Remarks}
\Crefname{rem}{Remark}{Remarks}

\newcommand{\boxedeqn}[1]{%
	\begin{equation}\fbox{%
		\addtolength{\linewidth}{-2\fboxsep}%
		\addtolength{\linewidth}{-2\fboxrule}%
		\begin{minipage}{\linewidth}%
			\begin{equation}#1\\begin{equation}+5mm]\end{equation}%
		\end{minipage}%
	}\end{equation}%
}

\providecommand{\norm}[1]{\|#1\|}

\providecommand{\grad}{\nabla}

\providecommand{\lip}{\beta}
\newcommand{\shift}{q}

\providecommand{\RR}{\mathbb{R}}

\providecommand{\opint}[1]{\left]#1\right[}

\providecommand{\ran}{\operatorname{ran}}

\providecommand{\parl}{\operatorname{par}}

\newcommand{\fix}{\ensuremath{\operatorname{Fix}}}

\providecommand{\Id}{\operatorname{{ Id}}}

\providecommand{\fady}{\varnothing}

\providecommand{\NN}{\mathbb{N}}

\providecommand{\fix}{\operatorname{Fix}}
\providecommand{\ran}{\operatorname{ran}}

\providecommand{\Id}{\operatorname{Id}}

\providecommand{\zer}{\operatorname{zer}}

\providecommand{\fady}{\varnothing}

\newcommand{\cran}{\ensuremath{\overline{\operatorname{ran}}\,}}

\providecommand{\ri}{\operatorname{ri}}

\providecommand{\RR}{\mathbb{R}}
\providecommand{\NN}{\mathbb{N}}

\definecolor{mybrown}{RGB}{120,75,50} 

\newcommand{\mybluebox}[1]{%
  \begingroup
  \setlength{\fboxsep}{6pt}
  \fcolorbox{mybrown}{mybrown!12}{$\displaystyle #1$}%
  \endgroup
}

\allowdisplaybreaks 
	
\begin{document}
										
										
										
										%

\author{ 
  \small Walaa M. Moursi\thanks{Department of Combinatorics and Optimization, University of Waterloo, Waterloo, Ontario N2L~3G1, Canada. E-mail: \texttt{walaa.moursi@uwaterloo.ca}}
  \and
  \small Andrew Naguib\thanks{Department of Combinatorics and Optimization, University of Waterloo, Waterloo, Ontario N2L~3G1, Canada. E-mail: \texttt{anaguib@uwaterloo.ca}}
  \and
  \small Viktor Pavlovic\thanks{Department of Combinatorics and Optimization, University of Waterloo, Waterloo, Ontario N2L~3G1, Canada. E-mail: \texttt{viktor.pavlovic@uwaterloo.ca}}
  \and
  \small Stephen A. Vavasis\thanks{Department of Combinatorics and Optimization, University of Waterloo, Waterloo, Ontario N2L~3G1, Canada. E-mail: \texttt{vavasis@uwaterloo.ca}}
}

										\title{\textsf{\textbf{Accelerated Proximal Gradient Methods\\ in the affine-quadratic case:\\ Strong convergence and limit identification}
											}
										}
										\date{\today}
										\maketitle
\begin{abstract}
Recent works by Bo{\c t}-Fadili-Nguyen \cite{BFN2025} and by Jang-Ryu \cite{JangRyu2025,RyuX2025} address the long–standing question of iterate convergence for accelerated (proximal) gradient methods. Specifically, Bo{\c t}-Fadili-Nguyen proved weak convergence of the discrete accelerated gradient descent (AGD) iterates and, crucially, convergence of the accelerated proximal gradient (APG) method in the composite case, in infinite–dimensional Hilbert spaces; their note also documents the announcement timeline. In parallel, Jang-Ryu established point convergence both for the continuous–time accelerated flow and for the discrete AGD method in finite dimensions, with the sequence of results summarized in \cite[Section~1.2]{JangRyu2025} and in the posts \cite{RyuX2025}.
These results leave unanswered the question of which minimizer is the limit point. We show in the \emph{affine-quadratic setting}: starting
from the same initial point, the difference between the PGM and APG iterates converges \emph{weakly} to zero, so APG converges weakly to the best approximation of the starting point in the solution set; moreover, under mild conditions, APG converges strongly.
 Our results are tight: a two–dimensional example shows that this coincidence of limits is specific to the affine–quadratic regime and does not extend in general.

\end{abstract}
										{ 
										}

\section{Introduction}

Throughout, we assume that 
\begin{empheq}[box=\mybluebox]{equation}
\text{$X$ is a real Hilbert space with inner product $\scal{\cdot}{\cdot}\colon X\times X\to\RR$, and induced norm $\|\cdot\|$.}
\end{empheq}
We also assume:
\begin{subequations}\label{intro}
\begin{empheq}[box=\mybluebox]{align}
& f : X \to \mathbb{R} \ \text{is convex and $\lip$-smooth, where } \lip \in \mathbb{R}_{++}, \label{30.1a} \\
& g : X \to \left]-\infty, +\infty\right] \ \text{is convex lower semicontinuous and proper}, \label{30.1b} \\
& F := f + g. \label{30.1c}
\end{empheq}
\end{subequations}
We study the problem
\begin{equation}
\minimize{x\in X}\ F(x):=f(x)+g(x).
\end{equation}
We set
\begin{equation}
\label{30.1d}
S := \argmin F,\qquad \mu := \min F(X).
\end{equation}
The PGM to solve \cref{intro} iterates the operator $T$ with stepsize $1/\lip$ defined by
\begin{equation}
\label{30.1e}
T := \prox_{\frac{1}{\lip} g}\!\left(\Id - \tfrac{1}{\lip}\nabla f\right).
\end{equation}
When $g\equiv 0$, PGM reduces to the gradient descent method.
\begin{example}[Method of Alternating Projections (MAP) as a PGM iterate.]
\label{ex:MAP}
Suppose that $U$ and $V$ are closed convex subsets of  $X$,
that\footnote{Let $C$ be a nonempty closed convex subset of $X$
  and let $x\in X$.
 Here and elsewhere, 
 we use $\iota_C$ to denote the \emph{indicator function} of $C$ defined by
$ \iota_{C}(x)=0$, if $x\in C$; and $\iota_{C}(x)=+\infty$, otherwise; 
 we use $P_Cx$ to denote the \emph{orthogonal projection} of $x$ onto the set $C$
 and we use $\dist{x}{C}:=\norm{x-P_Cx}$
to denote the \emph{distance of $x$ to $C$}.}
$f=\tfrac{1}{2}\distsq{\cdot}{V}$ and that $g=\iota_U$.
Then 
$\prox_g=P_U$, by, e.g.,  \cite[Example~23.4]{BC2017} and 
$\grad f=\Id-P_V$, by, e.g., 
\cite[Example~on~page~286]{Moreau1965} or 
\cite[Corollary~12.30]{BC2017}. Therefore, $\grad f$ 
is firmly nonexpansive by, e.g., \cite[Equation~1.7~on~page~241]{Zara71}, hence $1$-Lipschitz and 
the proximal-gradient operator in \cref{30.1e}
reduces to 
\begin{equation}
    T=P_UP_V,
\end{equation}
i.e., the MAP operator.
\end{example}

Let $x_0\in X$
and set $y_0:=x_0$. 
The accelerated schemes for the PGM generate two sequences 
$(x_k)_\knn$ and $(y_k)_\knn$ from a parameter sequence $(t_k)_{k\ge 0}$ via
\begin{equation}
x_{k+1}=T y_k,\qquad
y_{k+1}=x_{k+1}+\frac{t_k-1}{t_{k+1}}\,(x_{k+1}-x_k)\quad(\forall\knn),
\end{equation}
with the classical choices of $(t_k)_\knn$; see \cite{Nest83,BT2009,ChD15}. 
In the smooth case $g\equiv 0$, this iteration specializes to AGD by Nesterov \cite{Nest83}; for general $g$, it is the \emph{Fast Iterative Shrinkage--Thresholding Algorithm (FISTA)} of Beck and Teboulle \cite{BT2009}.  See also earlier results of Nesterov that generalize AGD to nonsmooth objective functions \cite{nesterov2004introductory}.
When $S\neq\varnothing$, these methods guarantee the optimal function-value rate
\begin{equation}
F(x_k)-\mu=\mathcal{O}\!\left(\tfrac{1}{k^{2}}\right).
\end{equation}

However, the \emph{convergence of the iterates} 
in discrete time has long been open:
for AGD as well as for APG, e.g., FISTA, convergence of $(x_k)_\knn$ to a point of $S$ under the classical assumptions on the parameter sequence $(t_k)_\knn$ was not settled for many years; convergence was known under slightly more damped rules (while preserving the $\mathcal{O}(1/k^2)$ rate), see, e.g., the scheme by Chambolle and Dossal \cite{ChD15}, and related inertial forward–backward analyses established convergence-rate statements for associated energies \cite{AttouchCabot2018}.

\paragraph{Continuous-time viewpoint, concurrent and very recent work.}
The accelerated flow (ODE model) provides a continuous-time surrogate that captures the $\mathcal{O}(1/t^{2})$ decay and underlies many Lyapunov constructions; see, e.g., \cite{SuBoydCandes2016,AttouchCabot2018}. 
Very recently, Jang-Ryu announced AI-assisted proofs establishing point convergence for the continuous-time flow and, shortly after, for the discrete Nesterov AGD in finite dimensions; for the precise timeline and context, see \cite[Section~1.2]{JangRyu2025} and the posts in \cite{RyuX2025}. 
Independently and in parallel, Bo{\c t}–Fadili–Nguyen proved weak convergence of the discrete AGD iterates and, crucially, extended their analysis to the composite setting to obtain convergence of APG schemes, including FISTA; see \cite[Theorem~1 and Section~2.4]{BFN2025}. 
Their manuscript also records the timeline of announcements and documents additional contributions beyond the finite-dimensional case.

\paragraph{Contribution.} 
In this paper, we analyze APG schemes in the case where $f$ is a quadratic function and $g$ is the indicator function
of a closed affine subspace, see \cref{eq:affineprob} below, assuming that a solution exists. 
Our results complement the recent works by Bo{\c t}--Fadili--Nguyen \cite{BFN2025} and by Jang-Ryu \cite{JangRyu2025}
as follows.
\begin{enumerate}[label=\fbox{R\arabic*},ref=R\arabic*]
\item \label{itm:R1}
Suppose that the parameter sequence is chosen such that the iterates of 
APG stay bounded.
We show that the sequence of iterates formed by the difference between the PGM sequence and the APG sequence generated by the same starting point converges weakly to $0$.
As a consequence, we identify the weak limit of the APG iterates to be the closest 
point to the starting point in $S$ (see \cref{thm:w:conv} below).

\item \label{itm:R2}
Under mild assumptions, we prove the strong convergence of APG schemes.
In this case, our assumptions on the parameter sequence are strikingly general: simply
any sequence $(t_k)_\knn$ that guarantees $F(x_k)\to \mu$
(see \cref{thm:st:conv} below).

\item \label{itm:R3}
A careful look at the conclusion of 
\cref{itm:R1} and \cref{itm:R2}, in view of \cref{lem:proxgradlim} below, reveals that APG schemes
have the same limit as PGM when applied 
with the same starting point
to solve \cref{eq:affineprob} below.
In  contrast, we show that this behaviour does not extend beyond the affine–quadratic setting,
by means of a counterexample in $\RR^2$; see \cref{prop:example-w2} below.
\end{enumerate}

\begin{remark}
 To our knowledge, there is no prior record of (a) identifying the weak limit of APG and (b) proving strong convergence of its iterates when the PGM operator \cref{30.1e} is affine.
\end{remark}

{\bf Organization and Notation.} 
In \cref{sec:affine} we develop a general framework for the iterates of a APG scheme when the associated PGM operator is affine.
In \cref{sec:bdd},
we provide a proof of the boundedness of APG iterates.\footnote{We note that Bo{\c t}–Fadili–Nguyen and Jang–Ryu gave the first proofs of boundedness for FISTA iterates; the former treats the possibly infinite–dimensional setting. Our proof, presented in \cref{sec:bdd} and included for completeness, was developed independently during the preparation of this manuscript.}
In \cref{sec:weak} and \cref{sec:strong}
 we provide our weak and strong convergence results for APG with the identified limits.
In \cref{sec:example} we provide a limiting example 
which shows that the limit agreement does not generalize beyond the affine–quadratic setting.
Finally, we provide a numerical application to image recovery.
Our notation is standard and follows largely \cite{BC2017}.

\section{A general framework for affine nonexpansive operators}
\label{sec:affine}
Let $\shift\in X$.
In this section, we assume that\footnote{Let $T:X\to X$ and let $(x,y)\in X\times X$. Then 
$T$ is \emph{nonexpansive} if $\norm{Tx-Ty}\le \norm{x-y}$; and $T$ 
is \emph{averaged nonexpansive} if there exists a nonexpansive operator $N:X\to X$
and a constant $\alpha\in \opint{0,1}$ such that $T=(1-\alpha )\Id+\alpha N$.}

\begin{equation}
\label{eq:L:assm}
\text{$L:X\to X$ is linear and nonexpansive,}
\end{equation}
 and we set
\begin{equation}
  \label{eq:def:T}
T:X\to X: x\mapsto Lx+\shift.
\end{equation}
\begin{remark}
\label{rem:nonnon}
It is straightforward to see that 
$T$ is nonexpansive (respectively, averaged nonexpansive)
if and only if 
$L$ is nonexpansive (respectively, averaged nonexpansive).
\end{remark}
Let $(x,y)\in X\times X$. It is clear that
\begin{equation}
\label{eq:obvious}
    T(x+y)=Tx+Ly=Ty+Lx.
\end{equation}
Rearranging yields 
\begin{equation}
\label{eq:obvious2}
    Lx=T(x+y)-Ty.
\end{equation}
Suppose that $U$ is a closed affine subspace of $X$. We use the notation
\begin{equation}
    \parl U:=U-U
\end{equation}
to denote the parallel subspace of $U$. 

We recall the following key fact.
\begin{fact}
\label{fact:Caffine}
Let $C$ be a nonempty closed convex subset of $X$,
let $U$ be a closed affine subspace of $X$,
 and let $(x,y)\in X\times X$.
 Then the following hold.
 \begin{enumerate}
\item 
\label{fact:Caffine:i}
$P_{y+C}=y+P_C(\cdot-y)$.
\item 
\label{fact:Caffine:ii}
Suppose that $U$ is a linear subspace of $X$.
Then $P_U$ is linear.
Moreover, $p=P_Ux$ $\siff$ $(p,x-p)\in U\times U^\perp$.
\item 
\label{fact:Caffine:iii}
$P_U$ is affine.
In particular, $P_U=P_U0+P_{\parl U}$.
\end{enumerate}
\end{fact}
\begin{proof}
\cref{fact:Caffine:i}:
This is \cite[Proposition~3.19]{BC2017}.
\cref{fact:Caffine:ii}:
This is \cite[Proposition~3.24(i)]{BC2017}.
\cref{fact:Caffine:iii}:
This is \cite[Proposition~3.22(ii)]{BC2017}
combined with \cref{fact:Caffine:i}.
\end{proof}

The following theorem is fundamental in our proofs of convergence.
\begin{theorem}
\label{thm:mainconv}
Let $W$ be a closed affine subspace of $X$,
let $(x_k)_\knn$ 
and $(p_k)_\knn$ be sequences in $X$, 
 and let $(s_k)_\knn$ be a sequence in $(\parl W)^\perp$.
 Suppose that $(\forall \knn)$
 \begin{equation}
 \label{eq:seq:xps}
x_k=p_k+s_k.
 \end{equation}
 Then the following hold.
 \begin{enumerate}
\item 
\label{thm:mainconv:0}
$\norm{s_k}\le \dist{x_k}{W}+\dist{p_k}{W}$.
 \item 
\label{thm:mainconv:-1}
Suppose that 
$\dist{x_k}{W}\to 0$
and $\dist{p_k}{W}\to 0$.
Then $s_k\to 0$.
If, in addition, $p_k\to p^*$,
then 
$x_k\to p^*$.

\item 
\label{thm:mainconv:i}
Suppose that $(p_k)_\knn$ is bounded.
Then $(x_k)_\knn$ is bounded $\siff $
$(s_k)_\knn$ is bounded.
\item 
\label{thm:mainconv:ii}
Suppose that $(x_k)_\knn$ and  $(p_k)_\knn$ are bounded
and  that every weak cluster point of 
$(x_k)_\knn$ and of $(p_k)_\knn$ lies in $W$.
Then the following hold.
\begin{enumerate}
\item
\label{thm:mainconv:iia}
$(s_k)_\knn$ converges weakly to $0$.
\item 
\label{thm:mainconv:iib}
If, in addition, $(\exists p^*\in W)$ such that $p_k\weakly p^*$
then 
\begin{equation}
x_k\weakly p^*.
\end{equation}
\end{enumerate}
 \end{enumerate}
\end{theorem}

\begin{proof}
\cref{thm:mainconv:0}:
Indeed, let $\knn$ and  observe that,
by \cref{fact:Caffine}\cref{fact:Caffine:iii},
applied with $U$ replaced by $W$,
$\Id-P_W=-P_W0+\Id-P_{\parl W}$.  
Now, applying
\cref{eq:obvious2} with 
$(x,y,T,L)$
replaced by $(s_k,p_k,\Id-P_W, \Id-P_{\parl W})$,
 and recalling that $(s_k)_\knn$ lies in $(\parl W)^\perp$
yield
\begin{equation}
s_k=s_k-P_{\parl W}s_k=(\Id-P_{\parl W} )s_k=(\Id-P_W)(x_k)-(\Id-P_W)(p_k).
\end{equation}
The triangle inequality yields
\begin{equation}
\norm{s_k}\le \norm{(\Id-P_W)(x_k)}+\norm{(\Id-P_W)(p_k)}
=\dist{x_k}{W}+\dist{p_k}{W}.
\end{equation}

\cref{thm:mainconv:-1}:
This follows from 
\cref{thm:mainconv:0} and \cref{eq:seq:xps}.

\cref{thm:mainconv:i}:
This is clear.
\cref{thm:mainconv:iia}:
Let $\overline{s}$ be a weak cluster point of
$(s_k)_\knn$, say $s_{n_k}\weakly \overline{s}$.
On the one hand, because $(\parl W)^\perp$
is closed we have $\overline{s}\in (\parl W)^\perp$.
On the other hand, and after dropping to a subsubsequence and relabelling if needed, we can and do 
assume that $p_{n_k}\weakly \overline{p}$ and 
$x_{n_k}\weakly \overline{x}$.
By assumption $(\overline{p},\overline{x})\in W\times W$. Therefore,
we have
\begin{equation}
  {(\parl W)}^\perp \ni\overline{s}\leftarrow s_{n_k}=x_{n_k}-p_{n_k}\weakly \overline{x}-\overline{p}\in W-W=\parl W.
\end{equation}
That is, $\overline{s}=0$. Since $\overline{s}$ was an arbitrary weak cluster point 
we conclude that $s_k\weakly 0$.
\cref{thm:mainconv:iib}:
This is a direct consequence of \cref{thm:mainconv:iia}
in view of \cref{eq:seq:xps}.
\end{proof}

We now prove the following key lemma.
\begin{lemma}
\label{lem:foundation}
Recall \cref{eq:L:assm} and  \cref{eq:def:T}. We have
\begin{enumerate}
\item 
\label{lem:foundation:i}
$\fix L=(\cran(\Id-L))^\perp$.
\item 
\label{lem:foundation:ii}
Let $z\in \ran (\Id-L)$ and let $k\in \NN$.
Then $L^kz \in \ran (\Id-L)$.
\item 
\label{lem:foundation:iii}
$\fix T\neq \fady \siff \shift\in \ran(\Id-L)$.
\end{enumerate}
Suppose that $\fix T\neq \fady$.
Then we additionally have.
\begin{enumerate}[resume]
\item 
\label{lem:foundation:iv}
$\parl(\fix T)=\fix L$.
In particular, $\fix T=u_0+\fix L$,
where 
$u_0=P_{\fix T}0$.
Consequently, 
$P_{\fix T}=u_0+P_{\fix L}$.
\item 
\label{lem:foundation:v}
$\ran(\Id-L)=\ran (\Id-T)$.
\end{enumerate}
\end{lemma}
\begin{proof}
\cref{lem:foundation:i}: On the one hand, it follows from  
\cite[Fact~2.25(iv)]{BC2017} that 
  $(\cran(\Id - L))^\perp = \fix L^*$.
  On the other hand, \cite[Lemma~2.1]{BDH03} implies that
  $\fix L=\fix L^*$.
  Altogether, $\fix L=(\cran(\Id-L))^\perp$ as claimed.

  \cref{lem:foundation:ii}: 
  By assumption $(\exists u \in X)$ such that $z = (\Id - L) u$.
  Therefore, $L^k z = L^k(\Id - L) u = L^k u - L^{k+1} u = (\Id - L) L^k u
  \in \ran (\Id-L)$.
  
  \cref{lem:foundation:iii}: It follows from \eqref{eq:def:T}
that $\fix T\neq \fady$
$\siff$
$(\exists w\in X) $ $w=Tw$
$\siff$
$(\exists w\in X) $ $\shift=w-Lw$
$\siff$
$\shift\in \ran(\Id-L)$.
  
  \cref{lem:foundation:iv}: 
  $(\subseteq)$: 
  Let $z \in \parl(\fix T)$. Then $(\exists (\overline{x}, \widetilde{x}) \in \fix T\times \fix T) $ \; $z= \overline{x} - \widetilde{x}$. Moreover, $z=\overline{x} - \widetilde{x} = T(\overline{x}) - T(\widetilde{x}) = L(\overline{x}) - L(\widetilde{x}) = L(\bar{x} - \widetilde{x})=Lz$. Hence, $z\in \fix L$. 
$(\supseteq)$: 
 Let $\overline{x} \in \fix L$.
 It follows from \cref{lem:foundation:iii} that $(\exists u \in X)$ 
 such that $ \shift = u - L u$, and hence $u \in \fix T$. 
 Moreover, $u + \overline{x}=\shift+Lu+L\overline{x}=\shift+L(u+\overline{x})=T(u+\overline{x})$.
 Now
 \begin{equation}
\overline{x} = u + \overline{x} - u
\in \fix T-\fix T=\parl(\fix T).
 \end{equation}
That is, $\parl(\fix T)=\fix L$ as claimed. 
The remaining claims follow
from applying \cref{fact:Caffine}\cref{fact:Caffine:iii}
with $U $ replaced by $\fix T$.
  
  \cref{lem:foundation:v}: 
  Note that  $ \ran(\Id - T)=-\shift+\ran(\Id-L)$.
  Now combine this with  \cref{lem:foundation:iii}.
\end{proof}

\begin{proposition}
    \label{prop:main}
Let $(\alpha_k)_\knn$ be a sequence of real numbers.
Let $x_0\in X$ and set $r_0:=0\in X$.
Update via
\begin{subequations}
    \begin{align}
        x_{k+1}&=T(x_k+r_k)
        \\
        r_{k+1}&=\alpha_k(x_{k+1}-x_k).
    \end{align}
\end{subequations}
Define the sequence $(s_k)_\knn:=\big(\sum_{n=1}^kL^nr_{k-n}\big)_\knn$.
Let $\knn$.
The following hold.
\begin{enumerate}
    \item 
    \label{prop:main:i}
    $x_k=T^kx_0+\sum_{n=1}^kL^n
    r_{k-n}=T^kx_0+s_k$.
\end{enumerate}
    Suppose that $\fix T\neq \fady$. Then we have.
\begin{enumerate}[resume]
    \item 
    \label{prop:main:ii}
    The sequence $(r_k)_\knn$ lies in $\ran(\Id-L)$.
    \item 
    \label{prop:main:iii}
    The sequence $(s_k)_\knn$ lies in $\ran(\Id-L)$.
    \item 
    \label{prop:main:iv}
    $P_{\fix T} x_k=P_{\fix T} T^kx_0$.
\end{enumerate}
\end{proposition}
\begin{proof}
\cref{prop:main:i}:
We proceed by induction.
The base case when $k=0$ is clear.
Now suppose for some $k\ge 0$ we have 
\begin{equation}
\label{eq:induc:aff}
    x_k=T^kx_0+\sum_{n=1}^kL^nr_{k-n}=T^kx_0+s_k.
\end{equation}
Then \cref{eq:induc:aff} 
and \cref{eq:obvious} applied 
with $(x,y)$ replaced by 
$\Big(T^kx_0, \sum_{n=1}^kL^nr_{k-n} +r_k\Big)$
yield
\begin{subequations}
    \begin{align}
        x_{k+1}&=T(x_k+r_k)
        \\
        &=T(T^kx_0)+ L\Big(\sum_{n=1}^kL^nr_{k-n}\Big)+Lr_k
        \\
        &=T^{k+1}x_0+ 
        \sum_{n=1}^kL^{n+1}r_{k-n}+Lr_k
         \\
        &=T^{k+1}x_0+ 
        \sum_{n=1}^{k+1}L^{n}r_{k+1-n}.
    \end{align}
\end{subequations}
Hence, the conclusion holds.

\cref{prop:main:ii}:
We proceed by induction. At $k=0$ 
 we have 
$r_0=0\in \ran(\Id-L)$.
Now suppose for some $k\ge 0$
we have $r_k\in \ran (\Id-L)$.
Using  
\cref{eq:obvious} applied with $(x,y)$
replaced by $(x_k,r_k)$,
\cref{lem:foundation}\cref{lem:foundation:ii}\&\cref{lem:foundation:v}
we have 
\begin{subequations}
    \begin{align}
r_{k+1}&=\alpha_k(x_{k+1}-x_{k})
=\alpha_k(Tx_{k}-x_{k}+Lr_k)
\\
&\in \alpha_k(\ran (T-\Id)+\ran (\Id-L))
=\alpha_k(-\ran (\Id-L)+\ran (\Id-L))
\\
&=\alpha_k\ran (\Id-L)=\ran (\Id-L). 
    \end{align}
\end{subequations}

\cref{prop:main:iii}:
Clearly $s_0=0\in \ran (\Id-L)$.
Let $k\ge n\ge 1$.
It follows from  \cref{lem:foundation}\cref{lem:foundation:v}
applied with $z$ replaced 
by 
$r_{k-n}$, in view of \cref{prop:main:ii}, that 
$L^nr_{k-n}\in \ran(\Id-L)$.
The conclusion follows from the fact that 
$ \ran(\Id-L)$ is a linear subspace.

\cref{prop:main:iv}:
Indeed, in view of \cref{prop:main:i}, 
\cref{lem:foundation}\cref{lem:foundation:iv}, and 
\cref{eq:obvious} 
applied with $(L,T)$ replaced by $(P_{\fix L},P_{\fix T})$,
we have 
\begin{equation}
    P_{\fix T}x_k=P_{\fix T}(T^kx_0+s_k)
   = P_{\fix T}(T^kx_0)+P_{\fix L}(s_k)
   = P_{\fix T}(T^kx_0).
\end{equation}
This completes the proof.
\end{proof}

A direct consequence of \cref{prop:main} is the following 
convergence result which is the basis of our weak convergence proof.
\begin{proposition}
\label{prop:conv}
Let $(\alpha_k)_\knn$ be a sequence of real numbers.
Let $x_0\in X$ and set $r_0:=0\in X$.
Let $(x_k)_\knn$ be defined as in \cref{prop:main}.
    Suppose that $\fix T\neq \fady$, that the sequence 
$(x_k)_\knn$ is bounded and its weak cluster points lie in $\fix T$
 and that $ T^kx_0\to P_{\fix T}x_0$. Then 
 \begin{equation}
    \text{$(x_k)_\knn$ converges weakly to $P_{\fix T}x_0$.}
 \end{equation}
\end{proposition}
\begin{proof}
Indeed, \cref{lem:foundation}\cref{lem:foundation:iv}\&\cref{lem:foundation:i}
implies that $\parl(\fix T)=\fix L=(\cran(\Id-L))^\perp$.
Now combine this with \cref{thm:mainconv}\cref{thm:mainconv:iib}
applied with
$(p_k)_\knn$ replaced by $(T^kx_0)_\knn$,
$W $ replaced by $\fix T$
 and $p^*$ replaced by $P_{\fix T}x_0\in \fix T$,
 in view of  \cref{prop:main} \cref{prop:main:i}.
\end{proof}

We now turn to the fundamental building blocks of the strong convergence result.
\begin{proposition}
\label{prop:aux}
Let $H_1, H_2$ be Hilbert spaces and let $A:H_1\to H_2$ be linear and continuous.
Suppose that $\operatorname{ran}A$
is closed. Then there exists $C\ge 0$ such that $(\forall x\in H_1)$
\begin{equation}
\dist{x}{\ker A}\ \le\ C\,\|Ax\|.
\end{equation}
\end{proposition}

\begin{proof}
Let  $x\in  H_1$ and let $(\overline{x},x^{\perp})\in \ker A\times (\ker A)^\perp $
such that 
$x=\overline{x}+x^{\perp}.$
Observe that  $Ax=A(\overline{x}+x^{\perp})=Ax^{\perp}$.
Hence, $x-P_{\ker A}x=\overline{x}+x^{\perp}-\overline{x}=x^{\perp}$,
and therefore
\begin{equation}
\label{eq:binv1}
\dist{x}{\ker A}= \|x^{\perp}\|.
\end{equation}
Now consider the restriction $A_\perp:(\ker A)^\perp\to\operatorname{ran}A$.
We claim that 
\begin{equation}
\text{$A_\perp$ is a bijection from 
$(\ker A)^\perp$
to $\operatorname{ran}A$}. 
\end{equation}
Indeed, $A_\perp$ is 
injective because $A_\perp x=0$ $\siff$ $x\in (\ker A)^\perp\cap\ker A=\{0\}$. 
Moreover,  let $y\in\operatorname{ran}A$. Then
$(\exists u\in X)$ with $Au=y$.
Set $w=u-P_{\ker A}u\in(\ker A)^\perp$.
Then, $Aw=A(u-P_{\ker A}u)=y$.
Hence, $A_\perp$ is surjective.
Since $\operatorname{ran}A$ is closed, $A_\perp$ is a bounded bijection between
two Banach spaces, hence $A_\perp^{-1}$ is bounded by the Bounded Inverse Theorem.
Therefore
\begin{equation}
\label{eq:binv2}
x^{\perp}
= A_\perp^{-1}(Ax^{\perp})
= A_\perp^{-1}(Ax).
\end{equation}
Combining \cref{eq:binv1} and \cref{eq:binv2} yields
\begin{equation}
\dist{x}{\ker A}
=\|x^{\perp}\|
=\bigl\|A_\perp^{-1}(Ax)\bigr\|
\le \|A_\perp^{-1}\|\,\|Ax\|.
\end{equation}
Setting $C:=\|A_\perp^{-1}\|$ yields the claim.
\end{proof}

We conclude this section with the  following fact and the subsequent corollary, which play a critical role in the convergence analysis.
\begin{fact}
\label{fact:fun:linalg}
Let $Y$ be a real Hilbert space.
 Suppose that $U$ is a closed linear subspace of $X$,
that $A:X\to Y$ is linear and that $\ran A $ is closed.
Then the following are equivalent\footnote{Let $U$ and $V$ be closed linear subspaces of $X$.
The \emph{cosine of the Friedrichs angle} between $U$ and $V$ is
$c_F
:= \sup\Big\{\scal{u}{v}\ \Big|\ u\in U\cap (U\cap V)^{\perp},\ v\in V\cap (U\cap V)^{\perp},\ \norm{u}\le 1,\ \norm{v}\le 1\Big\}$.}.
\begin{enumerate}
\item 
\label{fact:fun:linalg:i}
$c_F(U,\ker A)<1$.
\item 
\label{fact:fun:linalg:ii}
$U+\ker A$ is closed.
\item 
\label{fact:fun:linalg:iii}
$A(U)$ is closed.
\end{enumerate}   
\end{fact}

\begin{proof}
\cref{fact:fun:linalg:i}$\siff$\cref{fact:fun:linalg:ii}:
This is \cite[Theorem~13]{Deutsch1995}.
\cref{fact:fun:linalg:ii}$\siff$\cref{fact:fun:linalg:iii}:
This is \cite[Corollary~15.36]{BC2017}.
\end{proof}

\begin{corollary}
\label{cor:fun:linalg}
Let $Y$ be a real Hilbert space.
 Suppose that $U$ is a closed linear subspace of $X$,
that $A:X\to Y$ is linear and that $\ran A $ is closed.
Suppose that $U+\ker A$ is closed.
  Then $(\exists C\ge 0)$ such that $(\forall x\in U)$
  \begin{equation}
    \dist{x}{\ker A \cap U}=\dist{x}{\ker A_{|U}}\le C\norm{Ax}.    
  \end{equation}
\end{corollary}
\begin{proof}
Clearly, $\ker A_{|U}=\ker A\cap U$.
It follows from \cref{fact:fun:linalg} that $\ran A_{|U}= A(U)$
is closed.
Now apply \cref{prop:aux} with $H_1$ replaced by $U$
 and $A$ replaced by $ A_{|U}$.
\end{proof}

\begin{remark}
\label{rem:fun:linalg}
Suppose that $U$ and $V$ 
are closed affine subspaces of $X$.
Then $(\exists (u^\sperp,v^\sperp))\in (\parl U)^\perp \times (\parl V)^\perp$
such that $U=u^\sperp+\parl U$ and $V=v^\sperp+\parl V$. 
Hence, $U+V=u^\sperp+v^\sperp+\parl U+\parl V$.
Consequently,
\begin{equation}
    \text{$U+V$ is closed\; $\siff$\; $\parl U+\parl V$ is closed\; $\siff$\; $c_F(\parl U,\parl V)<1$.}
\end{equation}
\end{remark}

\section{The boundedness of the iterates of APG}
\label{sec:bdd}

The convergence—and hence boundedness—of the APG iterates is established in \cite[Proposition~2]{BFN2025} for an appropriately general selection of the parameter sequence that recovers the critical cases of \cite{Nest83,ChD15,BT2009}. See also \cite[Lemma~8]{JangRyu2025} for the AGD case in finite dimensions. For completeness, we record below a weaker argument for the boundedness of the APG iterates when the parameter sequence satisfies \cref{30.3}, which we developed independently while preparing this manuscript. We make no claim of priority; full credit for the result belongs to the recent works cited above.

In this section, we follow the notation used in \cite[Chapter~30]{BBM2023}. 

A sequence $(t_k)_{k \in \mathbb{N}}$ in $\mathbb{R}_{++}$ is a parameter sequence for APG if the following hold for every $k \in \mathbb{N}$:
\begin{subequations}
\label{30.2}
\begin{align}
t_k &\ge \frac{k+2}{2} \ge 1 = t_0,  
\label{eq:FS:i}
\\
t_k^{2} &\ge t_{k+1}^{2} - t_{k+1}. 
\label{30.2:b}
\end{align}
\end{subequations}

\begin{fact}
\label{thm:30.4}
Recall our assumptions~\cref{intro}, \cref{30.1d} and \cref{30.1e}.
Let $(t_k)_{k \in \mathbb{N}}$ be a parameter sequence for FISTA, i.e., \cref{30.2} holds.
Suppose that $S\neq \fady$.
Let $x_0 \in X$ and set $y_0 := x_0$.
Given $k \in \mathbb{N}$, update via
\begin{subequations}\label{30.3}
\begin{align}
x_{k+1} &:= T y_k = \prox_{\frac{1}{\lip} g} \big( y_k - \tfrac{1}{\lip} \nabla f(y_k) \big), \label{30.3a} \\
y_{k+1} &:= x_{k+1} + \frac{t_k - 1}{t_{k+1}} \big( x_{k+1} - x_k \big). \label{30.3b}
\end{align}
\end{subequations}
Then
$0 \leq F(x_k) - \mu \leq \frac{2L \distsq{x_0}{S}}{(k+1)^{2}}
= \mathcal{O}\!\left(\tfrac{1}{k^{2}}\right).$
In particular,
\begin{equation}
\label{eq:300}
F(x_k) \to \mu.
\end{equation}
\end{fact}
\begin{proof}
This is \cite[Theorem~4.4]{BT2009}.
\end{proof}

\begin{fact}
\label{lem:29.2}
Recall \cref{intro} and \cref{30.1e}.
Let $(x,y)\in X\times X$, and set $y_{+} := T y$. Then
\begin{equation}\label{29.4}
F(x) - F(y_{+}) \ge \frac{\lip}{2} \|x - y_{+}\|^{2} - \frac{\lip}{2} \|x - y\|^{2}.
\end{equation}
\end{fact}
\begin{proof}
See, e.g., \cite[Lemma~29.2]{BBM2023}.
\end{proof}

The conclusion of the following lemma
appeared in the proof of 
\cite[Theorem~30.2]{BBM2023}. (See also \cite[Proposition~1]{BFN2025}.)
We include the proof in \cref{app:0} for the sake of completeness.
\begin{lemma}
\label{lem:key}
Let $x_0\in X$
 and let $(x_k)_{k\in\NN}$
 and $(y_k)_{k\in\NN}$ be the sequences 
 generated by the APG algorithm in \cref{30.3}.
 Suppose that $S\neq \fady$.
 Let $x^*\in \fix T=S$.
 We set\footnote{In passing, we point out that the sequence $(z_k)_\knn$ defined in \cref{eq:zk-def} is called the \emph{auxiliary sequence} in \cite{BFN2025}.}  
 \begin{subequations}\label{eq:305}
\begin{align}
\delta_k &:= F(x_k) - \mu = F(x_k) - F(x^*) \ge 0 \quad (k \ge 0), \label{eq:delta-def}\\
z_k &:= t_k y_k + (1 - t_k) x_k \quad (k \ge 1), \label{eq:zk-def}\\
\xi_{k+1}&:=t_k^{2}\,\delta_{k+1}
+\tfrac{\lip}{2}\norm{z_{k+1}-x^*}^2\quad (k \ge 0).
\label{eq:def:xi}
\end{align}
\end{subequations}
 
Let $k\in\NN$.
Then\footnote{In passing, we acknowledge that
the inequalities in \cref{e:ineq1} are explicitly stated inside the proof of \cite[Theorem~30.2]{BBM2023}.} 
 \begin{equation}
 \label{e:ineq1}
0\le \ldots\le \xi_{k+1}\le \xi_{k}\le \ldots \le \xi_1\le \frac{\lip}{2}
\norm{x_0-x^*}^2.
 \end{equation}
 Moreover, the following hold.
 \begin{enumerate}
     \item 
     \label{lem:key:i}
     $(\xi_k)_{k\in\NN}$
 converges.
 \item 
 \label{lem:key:ii}
$(\norm{z_{k}-x^*}^2)_{k\ge 1}$
is bounded. More precisely,
$(\forall k\ge 1)$
 \begin{equation}
\norm{z_k-x^*}\le 
\norm{x_0-x^*}.
 \end{equation}
 In particular, $(z_k)_{k\in\NN}$ is bounded.
\end{enumerate}
\end{lemma}
\begin{proof}
See \cref{app:0}.
\end{proof}

We are now ready to prove the key  result 
in this section, namely the boundedness of  the iterates of APG. 
Let $(x,y)\in X\times X$.
We will make use of the following
inequality:
\begin{equation}
\label{ineq:cute}
2\scal{y-x}{y}\ge \norm{y}^2-\norm{x}^2.
\end{equation}

\begin{theorem}
\label{thm:bddfista}
Suppose that $S\neq \fady$.
Let $x_0\in X$
 and let $(x_k)_{k\in\NN}$
 and $(y_k)_{k\in\NN}$ be the sequences generated by the APG algorithm, see \cref{30.3}.
 Let $x^*\in \fix T=S$.
 Then the following hold.
 \begin{enumerate}
\item
\label{thm:bddfista:i}
$(\forall k\in\NN)$ $\norm{x_{k}-x^*}\le \norm{x_0-x^*}$.
\item 
\label{thm:bddfista:ii}
The sequence $(x_k)_{k\in\NN}$ is bounded.
Moreover, every weak cluster point of 
$(x_k)_{k\in\NN}$ lies in $S$.
 \end{enumerate}
\end{theorem}

\begin{proof}
\cref{thm:bddfista:i}:
We proceed by induction on $k$.
The base case, when $k=0$ is clear.
Now suppose for some $k\in\NN$
that
\begin{equation}
\label{eq:induchyp}
\norm{x_{k}-x^*}\le \norm{x_0-x^*}.
\end{equation}
Now let $(z_k)_{k\in\NN}$
be defined as in \cref{lem:key}.
Then \cref{ineq:cute}, applied with $(x,y)$
replaced by $(x_{k}-x^*,x_{k+1}-x^*)$, in view of \cref{eq:FS:i}, yields
\begin{subequations}
    \begin{align}
    \norm{x_0-x^*}^2
    &\ge \norm{z_{k+1}-x^*}^2=\norm{t_{k+1} y_{k+1} +(1-t_{k+1} )x_{k+1} -x^* }^2 
    \\
    &=\norm{x_{k+1}-x^*+(t_{k}-1) (x_{k+1}-x_{k})  }^2 
    \\
    &=\norm{x_{k+1}-x^*}^2+(t_{k}-1)^2\norm{x_{k+1}-x_{k}}^2
    \nonumber
    \\
    &\quad+2(t_{k}-1)\scal{x_{k+1}-x^*}{x_{k+1}-x_{k}}
    \\
    &
    \ge 
    \norm{x_{k+1}-x^*}^2
    +(t_{k}-1)\big( \norm{x_{k+1}-x^*}^2-\norm{x_{k}-x^*}^2\big)
    \\
    &\quad +(t_{k}-1)^2\norm{x_{k+1}-x_{k}}^2
                \\
    &= 
    t_k\norm{x_{k+1}-x^*}^2
    -(t_{k}-1)\norm{x_{k}-x^*}^2+(t_{k}-1)^2\norm{x_{k+1}-x_{k}}^2.
    \end{align}
\end{subequations}
Rearranging 
yields
\begin{equation}
\label{eq:48}
        t_k\norm{x_{k+1}-x^*}^2
    +(t_{k}-1)^2\norm{x_{k+1}-x_{k}}^2
    \le 
    (t_{k}-1)\norm{x_{k}-x^*}^2+ \norm{x_0-x^*}^2.
\end{equation}

Using \cref{eq:48} and \cref{eq:induchyp} we conclude that
\begin{subequations}
    \begin{align}
   \norm{x_{k+1}-x^*}^2
   &\le 
   \tfrac{1}{t_k}\big((t_{k}-1)\norm{x_{k}-x^*}^2+ \norm{x_0-x^*}^2\big)
   \\
   &   \le\tfrac{1}{t_k}\big((t_{k}-1)\norm{x_0-x^*}^2+ \norm{x_0-x^*}^2\big)
   \\
   &=\norm{x_0-x^*}^2.
    \end{align}
\end{subequations}
\cref{thm:bddfista:ii}:
The
boundedness of $(x_k)_{k\in\NN}$ is a direct 
consequence of \cref{thm:bddfista:i}.
Now let $\overline{x}$
be a weak cluster point of
$(x_k)_{k\in\NN}$, say $x_{n_k}\weakly \overline{x}$.
The (weak) lower semicontinuity of $F$
and \cref{eq:300}
imply 
\begin{equation}
   \mu\le F(\overline{x})\le \underline{\lim}\ F(x_{n_k}) =\lim F(x_{n_k})=\lim F(x_{k})\to \mu.
\end{equation}
Therefore, $F(\overline{x})=\mu$ and hence $\overline{x}\in S$ as claimed.
\end{proof}

\section{Weak convergence in the affine-quadratic case}
\label{sec:weak}

Let $Y$ be a real Hilbert space, let $A:X\to Y$ be linear and continuous
and let $b\in Y$.
Throughout the rest of this paper
we assume that 
\begin{empheq}[box=\mybluebox]{align}
\label{assum:U}
\text{ $U$ is a closed affine subspace of $X$.}
\end{empheq}

In this section, we consider the problem of 
\begin{empheq}[box=\mybluebox]{align}
\label{eq:affineprob}
\text{minimize}\quad & \tfrac{1}{2}\,\|Ax-b\|^2
\nonumber\\
\text{subject to}\quad & x \in U .
\end{empheq}
We set
\begin{equation}
\label{eq:fgF}
   \text{$f=\tfrac{1}{2}\norm{A(\cdot)-b}^2$,\quad $g=\iota_U$,\quad $F=f+g$,\quad $S=\operatorname{Argmin} F$,} 
\end{equation}
 and observe that 
 \begin{equation}
     \nabla f=A^*(A(\cdot)-b).
 \end{equation}
 Hence,
 $\grad f$ is an affine operator and $\lip$-Lipschitz 
 continuous with a constant $\lip=\norm{A^*A}$.
 Moreover, $P_U$
 is also an affine operator by \cref{fact:Caffine}\cref{fact:Caffine:iii}.
In this case the proximal-gradient operator in \cref{30.1e}
reduces to
\begin{equation}
\label{eq:Tproxgrad}
T=P_U\big(\Id-\tfrac{1}{\lip}A^*(A(\cdot)-b)\big).
\end{equation}
In view of \cref{fact:Caffine}\cref{fact:Caffine:iii}, we write
$P_U=u_0+P_{\parl U}$ where $u_0=P_U0$.
By \cref{fact:Caffine}\cref{fact:Caffine:ii}, $P_{\parl U}$
is linear and \cref{eq:Tproxgrad} reduces to 
\begin{equation}
T=u_0+P_{\parl U}\Big(\Id-\tfrac{1}{\lip}A^*(A(\cdot)-b)\Big)
=\underbrace{u_0+\tfrac{1}{\lip}P_{\parl U}A^*b}_{=:q}
+\underbrace{P_{\parl U}\Big(\Id -\tfrac{1}{\lip}A^*A\Big)}_{=:L}.
\end{equation}

We will apply the results of \cref{sec:affine} by recalling \cref{eq:def:T}
and  setting 
\begin{equation}
\label{eq:L:q:def}
q:= u_0+\tfrac{1}{\lip}P_{\parl U}A^*b
\quad\text{and}\quad 
L:=P_{\parl U}(\Id-\tfrac{1}{\lip}(A^*A)).
\end{equation}

The following two  lemmas are part of the literature.
We include their short proofs for the sake of completeness.
\begin{lemma}
\label{lem:FBfacts}
Let $T$ be defined as in \cref{eq:Tproxgrad}
 and let $L$ be defined as in \cref{eq:L:q:def}.
Then the following hold.
\begin{enumerate}
\item 
\label{lem:FBfacts:i}
$S=\fix T$.
\item
\label{lem:FBfacts:ii}
$T$ is an averaged nonexpansive operator.
\item 
\label{lem:FBfacts:iii}
$L$ is an averaged nonexpansive operator.
\end{enumerate}
\end{lemma}

\begin{proof}
  \cref{lem:FBfacts:i}:
This is \cite[Corollary~27.3(viii)]{BC2017}.
  \cref{lem:FBfacts:ii}:
This is \cite[Proposition~26.1(iv)(d)]{BC2017}.
\cref{lem:FBfacts:iii}:
Combine   \cref{lem:FBfacts:ii}
 and \cref{rem:nonnon}.
\end{proof}

\begin{lemma} 
\label{lem:proxgradlim}
Suppose that $S\neq \fady$.
Let $p_0\in X$  and $(\forall \knn)$
update via
\begin{equation}
p_{k+1}=Tp_k.
\end{equation}
Then $(p_k )_\knn$ converges strongly to $P_{S}p_0$.
\end{lemma}
\begin{proof}
It follows from \cref{lem:FBfacts}\cref{lem:FBfacts:i}\&\cref{lem:FBfacts:ii}
that $T$ is averaged nonexpansive and $\fix T\neq \fady$.
Consequently, $p_{k+1}-p_k\to 0$ by \cite[Proposition~5.16]{BC2017}, i.e.,
$T$ is asymptotically regular.
Now combine this with \cite[Theorem~3.3]{BLM17}
in view of \cref{lem:FBfacts}\cref{lem:FBfacts:i}.
\end{proof}

\begin{fact}
\label{fact:fistaBR}
Let $(t_k)_{k \in \mathbb{N}}$ in $\mathbb{R}_{++}$ be a sequence 
of parameters and set $(\alpha_k)_\knn=\big(\tfrac{t_{k}-1}{t_{k+1}}\big)$.
Let $T$ be defined as in \cref{eq:Tproxgrad} and 
let $(x_k)_\knn$
 and $(r_k)_\knn$ be defined as in \cref{prop:main}.
Suppose that $t_0=1$ and  one of the following holds.
\begin{enumerate}
\item 
\label{fact:fistaBR:i}
$(\forall \knn)$ $t^{2}_{k+1}-t^2_k=(1-\theta) t_{k+1}+\theta t_k$,\;  $\theta \in \left[0,1\right[$.
\item 
\label{fact:fistaBR:ii}
$(\forall \knn)$ $t_{k+1}=1+\tfrac{k}{\alpha-1}$ and $\alpha\ge 3$.
\item 
\label{fact:fistaBR:iii}
$(\forall \knn)$ $t^{2}_{k+1}-t^2_k\le  t_{k+1}$.
\end{enumerate}
Then $(x_k)_\knn$ converges weakly to a point in $S$.
In particular, $(x_k)_\knn$ is bounded.
\end{fact}
\begin{proof}
\cref{fact:fistaBR:i}:
This is \cite[Theorem~1~and~Section~2.4]{BFN2025}.
See also \cite[Theorem~14]{JangRyu2025} for the case $\theta=0$
and $X$ is finite-dimensional.
\cref{fact:fistaBR:ii}:
The critical case $\alpha=3$ follows from \cref{fact:fistaBR:i}
applied 
with $\theta=\tfrac{1}{2}$.
See \cite{ChD15} for the case $\alpha>3$.
\cref{fact:fistaBR:iii}:
This is \cite[Remark~3(iii)]{BFN2025}.
\end{proof}

We are now ready for the main result in this section.

\begin{theorem}[weak convergence of APG iterates in the affine-quadratic case.]
\label{thm:w:conv}

    Recall \cref{eq:affineprob}
    and \cref{eq:fgF}. 
    Suppose that $S\neq \fady$.
    Let $(t_k)_{k \in \mathbb{N}}$ in $\mathbb{R}_{++}$ be a sequence 
of parameters and set $(\alpha_k)_\knn:=\big(\tfrac{t_{k}-1}{t_{k+1}}\big)$.
Let $T$ be defined as in \cref{eq:Tproxgrad} and 
let $(x_k)_\knn$
 be defined as in \cref{prop:main}.
 Suppose that $(x_k)_\knn$ is bounded and its weak cluster points lie in $S$
 (see \cref{fact:fistaBR} and \cref{thm:bddfista}\cref{thm:bddfista:ii} for sufficient conditions).
    Then 
    \begin{equation}
\text{$(x_k)_\knn$ converges weakly to $P_{S}x_0$.}
    \end{equation}  
\end{theorem}

\begin{proof}
Observe that $\fix T\neq \fady$ by \cref{lem:FBfacts}\cref{lem:FBfacts:i}. 
It follows from 
\cref{lem:proxgradlim}, applied with $p_0$ replaced by 
$x_0$,
that
\begin{equation}
    T^kx_0\to P_Sx_0.
\end{equation}
Now combine this with \cref{prop:conv} to learn that 
$(x_k)_\knn$ converges weakly to $P_{S}x_0$.
\end{proof}

\section{Strong convergence in the affine-quadratic case}
\label{sec:strong}

In this section, we prove the strong convergence of the 
APG iterates under the  assumptions that 
$\ran A$ is closed and that $\ker A+\parl U$ is closed; these assumptions
are automatic in finite dimensions. We 
emphasize that our argument accommodates a more general parameter sequence:
Indeed, we impose no additional structure on $(\alpha_k)_\knn$ associated 
with $(t_k)_\knn$ beyond the requirement that the resulting scheme guarantees $F(x_k)\to \mu$.
When $X$ is finite-dimensional, the conclusions of \cref{thm:w:conv} and \cref{thm:st:conv} below coincide; 
nevertheless, in the finite-dimensional case, \cref{thm:st:conv} is more general
given the generality of the parameter sequence.

Recall \cref{eq:affineprob}, \cref{eq:fgF} and \cref{eq:Tproxgrad}.
We start with the following key result,
which gives a formula for
$\parl S$ when $S\neq \fady$.

\begin{lemma}
\label{lem:parS}
Suppose $S\neq \fady$. Then 
\begin{equation}
\label{lem:FBfacts:iv}
\parl S=S-S=\fix L=\cran(\Id-L)^\perp =\ker A\cap (\parl U).
\end{equation}
\end{lemma}
\begin{proof}
  Combining \cref{lem:foundation}\cref{lem:foundation:iv}
   and \cref{lem:FBfacts:i} yields 
   $\parl S=S-S=\fix T-\fix T=\fix L=(\cran(\Id-L))^\perp$. We now turn to the last identity in 
\cref{lem:FBfacts:iv}.
Recalling \cref{eq:L:q:def}, 
in view of \cite[Proposition~3.4]{BM:found}
 and \cite[Fact~2.25(vi)]{BC2017}
we learn that
\begin{subequations}
    \begin{align}
        \cran (\Id-L)
        &=\overline{\ran(\Id-P_{\parl U})+\ran (\Id- (\Id-\tfrac{1}{\lip}(A^*A)))}
        \\
        &=\overline{(\parl U)^\perp+\cran (A^*A)}
        \\
        &=\overline{(\parl U)^\perp+\cran A^*}=\overline{(\parl U)^\perp+\ran A^*}.
   \end{align}
\end{subequations}
Therefore,  $(\cran (\Id-L))^\perp=(\overline{(\parl U)^\perp+\ran A^*})^\perp
=((\parl U)^\perp)^\perp\cap(\cran A^*)^\perp=\parl U\cap \ker A$
by, e.g., \cite[Fact~2.25(v)]{BC2017}.
This completes the proof.
\end{proof}

Below, we provide two instances where $S\neq \fady$.
\begin{lemma}
\label{lem:Snotfady}
   Let $T$ be defined as in \cref{eq:Tproxgrad}.
   Suppose that one of the following holds.
   \begin{enumerate}
\item 
\label{lem:Snotfady:i}
$X$ is finite-dimensional.
\item 
\label{lem:Snotfady:ii}
$V$ is a closed affine subspace of $X$
and that $f=\tfrac{1}{2}\distsq{\cdot}{V}$.
   \end{enumerate}
   Then $S\neq \fady$ and $\parl S=\ker A\cap (\parl U)$.
\end{lemma}

\begin{proof}
In view of \cref{lem:parS}, it suffices to show that
$S\neq \fady$.
    \cref{lem:Snotfady:i}:
    Recalling \cref{eq:affineprob}, 
Fermat's theorem, see, e.g.,  \cite[Theorem~16.3]{BC2017} implies that 
$S=\zer(\partial F)=\zer (\partial (f+\iota_U))=\zer (\grad f+N_U)$.
Hence, $S\neq \fady$ $\siff$ $0\in \ran (\grad f+N_U)=\ran(A^*(A(\cdot)-b)+N_U)$.
Because the sets $\ran(A^*(A(\cdot)-b)+N_U)$, $\ran N_U=(\parl U)^\perp$ 
and $\ran A^*(A(\cdot)-b)$
are affine we learn that\footnote{Let $S\subseteq X$. We use $\ri S$ to denote the \emph{relative interior} of $S$.} 
$\ri \ran(A^*(A(\cdot)-b)+N_U)=\ran(A^*(A(\cdot)-b)+N_U)$, 
$\ri \ran N_U= \ran N_U$ and $\ri \ran A^*(A(\cdot)-b)=\ran A^*(A(\cdot)-b)$.
Now combine this with \cite[Theorem~3.13]{BMW13} to learn that 
\begin{equation}
  \ran(A^*(A(\cdot)-b)+N_U)  
  =\ran (A^*(A(\cdot)-b))+(\parl U)^\perp.
\end{equation}
Finally, observe $0\in \ran (A^*(A(\cdot)-b))$ and $0\in (\parl U)^\perp$.
Consequently, $0\in  \ran(A^*(A(\cdot)-b)+N_U)  $ and $S\neq \fady$.
    \cref{lem:Snotfady:ii}:
It follows from  \cref{lem:FBfacts}\cref{lem:FBfacts:i} and 
\cref{ex:MAP}
that $S=\fix T=\fix P_UP_V$
which is nonempty by 
\cite[Theorem~5.6.1]{BBL97}.
\end{proof}

Let $(\alpha_k)_\knn$ be a sequence of real numbers.
Let $x_0\in X$ and set $r_0:=0\in X$.
Update via
\begin{subequations}
\label{eq:genscheme}
    \begin{align}
        x_{k+1}&=T(x_k+r_k)\in U
        \\
        r_{k+1}&=\alpha_k(x_{k+1}-x_k).
    \end{align}
\end{subequations}
\begin{proposition}
\label{lem:distkerA}
Suppose $S\neq \fady$.
Let $\overline{x}\in S$.
Then
the following hold.
\begin{enumerate}
    \item 
    \label{lem:distkerA:i}
    Let $k\ge 1$. Then 
    $\norm{A(x_k-\overline{x})}^2= 2(f(x_k)-f(\overline{x}))$.
    \item 
    \label{lem:distkerA:ii}
    Suppose that $\ran A $ is closed, that 
    $\ker A+\parl U$ 
    is closed\footnote{This assumption is true if, for instance, $X$ is finite-dimensional.
    See \cref{fact:fun:linalg} and \cref{rem:fun:linalg} for equivalent characterizations.} 
    and that $f(x_k)-f(\overline{x})\to 0$; equivalently $F(x_k)-\mu\to 0$. Then $\dist{x_k}{S}\to 0$.
\end{enumerate}
\end{proposition}
\begin{proof}
\cref{lem:distkerA:i}:
On the one hand, we note that $x_k-\overline{x}\in U-U=\parl U$.
On the other hand, by the optimality of $\overline{x}$
$(\exists u^\sperp\in (\parl U)^\perp)$ such that 
$u^\sperp=\grad f(\overline{x})=A^*A\overline{x}-A^*b$.
Hence $\scal{u^\sperp}{x_k - \overline{x}}=0$.
Now
\begin{subequations}
\begin{align}
2(f(x_k)-f(\overline{x})) &= \Vert Ax_k - b \Vert^2 - \Vert A\overline{x} - b \Vert^2 \\
&= \langle x_k, A^*A x_k \rangle - 2\langle x_k, A^* b \rangle - \langle \overline{x}, A^*A \overline{x} \rangle + 2\langle \overline{x}, A^* b \rangle \\
&= \langle x_k + \overline{x}, A^*A(x_k - \overline{x}) \rangle - 2\langle b, A(x_k - \overline{x}) \rangle \\
&= \langle A^*A(x_k + \overline{x}) - 2A^*b, x_k - \overline{x} \rangle \\
&= \langle 2A^*A \overline{x} - 2A^*b + A^*A(x_k - \overline{x}), x_k - \overline{x} \rangle \\
&= 2\langle u^\sperp, x_k - \overline{x} \rangle + \langle A^*A(x_k - \overline{x}), x_k - \overline{x} \rangle
\\
&=\langle A(x_k - \overline{x}), A(x_k - \overline{x}) \rangle
=\norm{A(x_k-\overline{x})}^2.
\label{eq:fdif}
\end{align}
\end{subequations}
\cref{lem:distkerA:ii}:
First note that, $(\forall k\ge 1)$ $ f(x_k)-f(\overline{x})= f(x_k)+\iota_U(x_k)-\mu=F(x_k)-\mu$.
This verifies the ``equivalently" part.
Now, observe that $\ker A_{|\parl U}=\ker A\cap \parl U$.
In view of 
\cref{lem:FBfacts:iv},
write $S=\overline{x}+(\ker A\cap \parl U)$.
It follows from 
\cref{cor:fun:linalg} applied with $U$ replaced by $\parl U$
and \cref{lem:distkerA:i}
that there exists $C\ge 0$
such that 
\begin{subequations}
\begin{align}
\distsq{x_k}{S}&=\distsq{x_k}{\overline{x}+(\ker A\cap \parl U)}
   \\
   &=\distsq{x_k-\overline{x}}{\ker A\cap \parl U}
   \\
   &=\distsq{x_k-\overline{x}}{\ker A_{|\parl U}}\\
   &\le C \norm{A_{|\parl U}(x_k-\overline{x})}^2
   =C \norm{A(x_k-\overline{x})}^2
   \\
   &= 2C(f(x_k)-f(\overline{x}))\to 0.
\end{align}   
\end{subequations}
The proof is complete.
\end{proof}

We are now ready to prove our main result.
\begin{theorem}[strong convergence of APG iterates in the affine-quadratic case]
\label{thm:st:conv}
    Recall \cref{eq:affineprob}, \cref{eq:fgF} and \cref{eq:Tproxgrad}. 
    Let $(x_k)_\knn$ be defined as in \cref{eq:genscheme}.
 Suppose that $F(x_k)-\mu\to 0$    
and  that 
    one of the following hold.
    \begin{enumerate}
    \item 
    \label{thm:st:conv:i}
    $S\neq \fady$, $\ran A$ is closed and $\parl U+\ker A$ is closed. 
    \item 
        \label{thm:st:conv:ii}
    $V$ is a closed affine subspace of $X$,
    $f=\tfrac{1}{2}\distsq{\cdot}{V}$ and $U+V$ is closed.
    \item 
        \label{thm:st:conv:iii}
    $X$ is finite-dimensional.
    \end{enumerate}
    Then 
    \begin{equation}
    \label{eq:strongconv}
x_k\to P_Sx_0.
    \end{equation} 
    In particular, if $(t_k)_\knn$ is a parameter sequence of APG, see, e.g., \cref{30.2}, and we set  
    $(\alpha_k)=\Big(\tfrac{t_{k}-1}{t_{k+1}}\Big)_\knn$
    then the APG sequences defined in \cref{thm:30.4}
    satisfy \cref{eq:strongconv}.
\end{theorem}

\begin{proof}
\cref{thm:st:conv:i}:
Indeed, set $p_0=x_0$
and consider the sequence 
$(p_k)_\knn$
generated by the update $(\forall \knn)$
$P_{k+1}=Tp_k$.
On the one hand, it follows from
\cref{lem:proxgradlim} 
that $p_k\to P_{S}x_0\in S$.
Hence, $\dist{p_k}{S}\to 0$.
On the other hand, 
it follows from \cref{lem:distkerA}\cref{lem:distkerA:ii}
that $\dist{x_k}{S}\to 0$.
Now combine this with \cref{thm:mainconv}\cref{thm:mainconv:-1}
applied with $(p^*, W)$ replaced by $(P_Sx_0,S)$
in view of 
\cref{lem:FBfacts}\cref{lem:FBfacts:i}.
\cref{thm:st:conv:ii}:
Observe that $S\neq \fady$ by \cref{lem:Snotfady}\cref{lem:Snotfady:ii}.
Moreover, setting $A=\grad f=\Id-P_V$ yields 
$\ran A=V^\perp$ which is closed and 
$\ker A=V$.
Clearly, $\parl U+V$ is closed $\siff $  $U+V$ is closed.
Now combine this with \cref{thm:st:conv:i}.
\cref{thm:st:conv:iii}:
It follows from \cref{lem:Snotfady}\cref{lem:Snotfady:i}
that $S\neq \fady$.
Moreover, $\parl U+\ker A$ is a finite-dimensional affine subspace, 
hence closed.
Now combine this with \cref{thm:st:conv:i}.
The ``in particular" part is obvious, in view of \cref{eq:300}, by observing that
\cref{eq:genscheme} with  $(\alpha_k)=\Big(\tfrac{t_{k}-1}{t_{k+1}}\Big)_\knn$
is the APG algorithm in \cref{thm:30.4}.
\end{proof}



\section{APG for a cone and an affine subspace}
\label{sec:example}
Consider the PGM and the APG
scheme in \cref{30.2},
applied to solve \cref{eq:affineprob}, 
which recovers FISTA, Chambolle-Dossal and AGD as special cases.
\cref{thm:st:conv} shows that in finite-dimensional Hilbert spaces, running both methods 
with the same starting point yields the same limit point. 
In \cref{prop:example-w2} below, we demonstrate that this result
does not generalize beyond the affine-quadratic setting in \cref{eq:affineprob}.

In this section, we set 
\begin{equation}
U=\{(x_1,x_2)\in\mathbb{R}^2: x_1+x_2=1\}\;\;\text{and}\;\; V=\mathbb{R}^2_{+},
\end{equation}
and we use the APG parameter sequence $(t_k)_\knn:=\big(\tfrac{k+2}{2}\big)_\knn$, which 
implies that 
\begin{equation}
(\alpha_k)_\knn:=\Big(\frac{t_k-1}{t_{k+1}}\Big)_\knn=\Big(\frac{k}{k+3}\Big)_\knn
\;\text{lies in }\;\left]0,1\right[.
\end{equation}

Let $(a,b)\in \RR^2$. Recall that
\begin{equation}
\label{eq:PVPU}
P_V(a,b)=\big(\max\{a,0\},\,\max\{b,0\}\big)\quad\text{and}
\quad
P_U(a,b)
=\tfrac{1}{2}(a-b+1,b-a+1).
\end{equation}


Let $w\ge 1$,  let $x_0:=(w,0)$ and let $y_0:=p_0:=x_0$. 
$(\forall \knn)$ update via
\begin{subequations}
    \begin{align}
        p_{k+1}&=P_UP_V(p_k)
        \\
         x_{k+1}&=P_UP_V(y_k)
         \\
        y_{k+1}&=x_{k+1}+\alpha_k(x_{k+1}-x_k).
    \end{align}
\end{subequations}

\begin{remark}
\label{rem:coneaff}
\begin{enumerate}
\item
\label{rem:coneaff:i}
    In view of \cref{ex:MAP} and \cref{thm:30.4},
the sequence $(p_k)_\knn$
(respectively, the sequences $(x_k)_\knn$ and $(y_k)_\knn$
    is the MAP sequence (respectively, the APG sequences)
    generated with the starting point $(w,0)$ to solve 
\begin{equation}
\label{eq:conceaff:prob}
\text{minimize}\quad  \tfrac{1}{2}\,\distsq{x}{V} \quad 
\text{subject to}\quad  x \in U .    
\end{equation}
\item 
\label{rem:coneaff:ii}
Recalling \cref{eq:conceaff:prob}, it is clear that 
$S=U\cap V=\menge{(u,1-u)\in \RR^2}{0\le u\le 1}\neq \fady$. 
On the one hand, using, e.g.,  
\cite[Example~28.11]{BC2017}, we learn that $(p_k)_\knn$
converges to a point in $U\cap V$.
On the other hand, the recent breakthrough in 
\cite[Section~2.4]{BFN2025} implies the convergence 
of $(x_k)_\knn$ (and hence, of $(y_k)_\knn$, because 
$(\alpha_k)_\knn$ lies in $\left[0,1\right [$\;)
to a point in $U\cap V$. In passing, we point out that we 
identify the limit of $(x_k)_\knn$ under an abstract condition as illustrated
in \cref{prop:FISTAcone}
below.
\end{enumerate}
\end{remark}
Recalling \cref{rem:coneaff}\cref{rem:coneaff:ii},
we set 
\begin{equation}
p^*:=\lim_{k\to\infty}p_k\quad \text{and} \quad x^*:=\lim_{k\to\infty}x_k=\lim_{k\to\infty}y_k. 
\end{equation}

The following result gives the limit of the sequence $(p_k)_\knn$.
\begin{proposition}
\label{prop:MAP-to-10}
We have $(p_k)_{k\ge 1}=\big(\big(1+\tfrac{w-1}{2^{k}},\ -\,\tfrac{w-1}{2^{k}}\big)\big)_{k\ge 1}\; $.
Hence,
\begin{equation}
\label{eq:MAP-limit}
p^*=(1,0).
\end{equation}
\end{proposition}

\begin{proof}
See \cref{app:A}.
\end{proof}

The following lemma 
is crucial in the proof 
of \cref{prop:FISTAcone} below.
\begin{lemma}
\label{lem:aux}
Let $M\ge 2$ and let $a\in\,]0,1[$.
Define a sequence $(a_k)_{k\ge M}$ by
$a_M:=a$ and $(\forall k\ge M)$
\begin{equation}
a_{k+1}:=\frac{k-1}{k+2}\,a_k.
\end{equation}
Let $k\ge M$.
Then the following hold.
\begin{enumerate}
\item
\label{lem:aux:i}
$a_k>0\quad\text{and}\quad a_{k+1}<a_k.$

\item
\label{lem:aux:ii}
$a_k
= a\prod_{j=M}^{k-1}\frac{j-1}{j+2}
= a\,\frac{(M-1)M(M+1)}{(k-1)k(k+1)}$.

\item
\label{lem:aux:iii}
$\lim_{k\to\infty}a_k=0$.
\item
\label{lem:aux:iv}
$\sum_{j=M}^{k} a_{j+1}
= \frac{a\,(M-1)}{2}\,\Bigl(1-\frac{M(M+1)}{(k+1)(k+2)}\Bigr)$.
Consequently,
$\sum_{j=M}^{\infty} a_{j+1}=\frac{a\,(M-1)}{2}$.

\item
\label{lem:aux:v}
$\sum_{j=k}^{\infty} a_{j+1}
= \frac{a\,(M-1)M(M+1)}{2\,k(k+1)}$.
\item
\label{lem:aux:vi}
$a_{k+1}\ \le\ \sum_{j=k+1}^{\infty} a_{j+1}$.

\end{enumerate}
\end{lemma}

 \begin{proof}
 See \cref{app:B}.
 \end{proof}

In \cref{prop:FISTAcone} below we provide an abstract condition that guarantees that $x^*\neq p^*$.

\begin{proposition}
\label{prop:FISTAcone}
Let $(u_k)_{k\ge 1}$ and $(d_k)_{k\ge 1}$ be the sequences  
of real numbers defined by $((1-u_k,u_k))_{k\ge 1}:= (x_k)_{k\ge 1}$
and $(d_k)_{k\ge 1}:=(u_k-u_{k-1})_{k\ge 1}$.
Suppose there exists $M\ge 3$ such that
\begin{equation} 
y_M\in U\cap V; 
\qquad d_M> 0,\qquad u^*:=u_M+\tfrac{1}{2}(M-1)d_M\in \left]0,1\right].
\end{equation}
Then 
\begin{equation} 
x^*=\bigl(1-u^*,\,u^*\bigr)\in U\cap V.
\end{equation}
In particular, 
$x^*\neq (1,0)=
p^*$.
\end{proposition}

\begin{proof}
See \cref{app:C}.
\end{proof}

We conclude this section with a concrete choice of $x_0$ 
to demonstrate our result.
\begin{example}
\label{prop:example-w2}
Let $x_0=(5,0)$. 
 Then 
 \begin{equation}
(1,0)=p^*\neq x^*=\Bigl(\tfrac{19}{32},\tfrac{13}{32}\Bigr).
 \end{equation}
\end{example}

\begin{proof}
See \cref{app:D}.
\end{proof}

\begin{figure}
    \centering
    \includegraphics[scale=0.55]{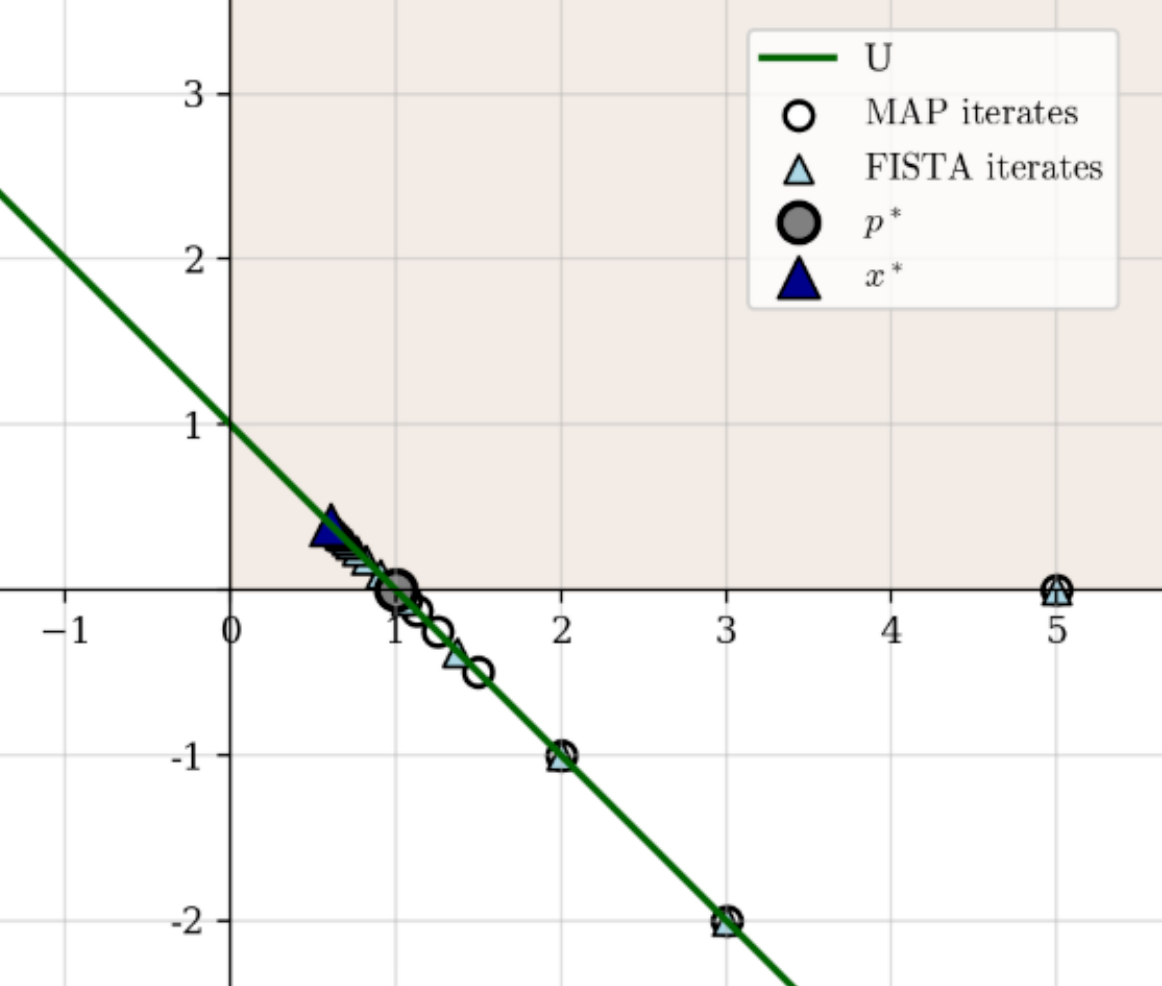}
    \caption{A \texttt{Python} plot illustrating \cref{prop:example-w2}.
    Shown are the initial point $(5,0)$, 
    the first $40 $ iterates of the MAP iterates $(p_k)_\knn$  and FISTA iterates $(x_k)_\knn$  along with the limits $p^*$
    and $x^*$.}
    \label{fig:cone_aff}
\end{figure}

\section{Numerical Experiments}
Suppose we are given a two $n \times n$ images which are similar, such as subsequent frames from a video, and one of them, say, $x^c$ is corrupted. 
A fraction of the pixels in $x^c$ are missing or blacked out. 
We only have access to a subset of the pixel values. Specifically, 
we know $x^c_{ij}$ for those indices $(i,j)$ belonging to a set 
$I \subseteq \{1,\dots,n\} \times \{1,\dots,n\}$. 
Let $p = |I|$ denote the number of known pixels. 
We formulate the objective function as 
\begin{equation}
    f(x) = \tfrac{1}{2}\|A x - b\|^2,
\end{equation}
where $b \in \RR^p$ is the vector of known pixel values, 
and $A \in \RR^{p \times n^2}$ is the sampling matrix obtained 
by selecting from the $n^2 \times n^2$ the rows of the identity matrix 
 corresponding to the known pixels.

Next, we introduce affine constraints through the discrete cosine transform (DCT)
\cite{DCT}. The DCT is closely related to the discrete Fourier transform (DFT)
but operates entirely in the real domain, expressing a discrete signal as a
linear combination of cosine waves with different frequencies. It is widely
used in signal processing \cite{dct_signal}, as well as in biometrics, medical
imaging, and image inpainting \cite{ochoa2019discrete}. The present setting is
a prototypical inpainting task. The DCT can be computed efficiently via the
fast Fourier transform (FFT) in $\mathcal{O}(n \log n)$ operations for
length-$n$ signals.

Formally, the DCT is a linear map represented by an orthogonal matrix $Q$.
Hence $Q^{\top}Q = QQ^{\top} = \Id$, and $Q^{\top}$ acts as the inverse DCT.
For images, we use the two-dimensional DCT obtained by applying the
one-dimensional DCT row-wise and then column-wise. Equivalently, the 2D-DCT
is the  Kronecker product $Q \otimes Q$ of the DCT matrix.

To form our constraints, we assume that alongside the corrupted image,
a partial discrete cosine transform (DCT) is available---for example,
the subset of coefficients corresponding to high-frequency components of the uncorrupted image.
Let $C \in \RR^{m \times n^2}$ denote the DCT submatrix consisting of
$m$ rows of the full $n^2 \times n^2$ DCT matrix corresponding to the
known frequency indices, and let $d \in \RR^m$ be the vector of the
corresponding known frequency coefficients. Together, these define the
affine constraint 
\begin{equation}
    U=\menge{x\in \RR^{n^2}}{Cx = d}.
\end{equation}
Let $x\in \RR^{n^2}$. It follows from the definition of $C$ that $CC\tran=\Id$. Therefore,
by e.g., \cite[Proposition~8.10]{BBM2023},
we have 
\begin{equation}
\label{eq:PUformula}
P_U(x) = x - C\tran(CC\tran)^{-1}(Cx - d)= x - C\tran(Cx - d).
\end{equation}

\begin{figure}[h!]
    \centering
    \subfloat[]{
        \includegraphics[width=0.3\textwidth]{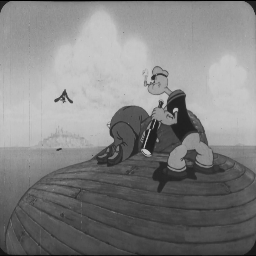}
        \label{fig:base_img}
    }
    \hfill
    \subfloat[]{
        \includegraphics[width=0.3\textwidth]{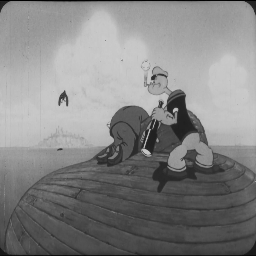}
        \label{fig:uncorr_img}
    }
    \hfill
    \subfloat[]{
        \includegraphics[width=0.3\textwidth]{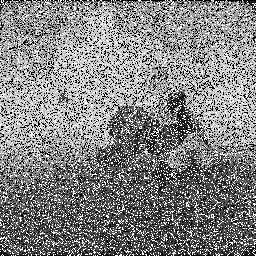}
        \label{fig:corr_img}
    }
    \caption{\cref{fig:base_img} is the preceding frame to the corrupted one we receive, \cref{fig:uncorr_img} is the uncorrupted frame (the one we try to recover), and 
\cref{fig:corr_img} is the corrupted version with $40\%$ of pixels
missing.}
    \label{fig:base_corr}
\end{figure}

For our numerical experiment,  we use two frames from the animated film ``Popeye the Sailor meets Sinbad the Sailor'' \cite{popeye}, with $n=256$, so
$x\in\RR^{n^2}=\RR^{65{,}536}$. We corrupt $40\%$ of the pixels in one of the frames, leaving
approximately $39{,}000$ known values. From these, we take $p=35{,}000$
randomly selected pixels, forming $A\in\RR^{p\times n^2}$ and
$b\in\RR^{p}$ for the objective. We additionally assume $m=2^{13}=8{,}192$
known DCT coefficients from the uncorrupted frame, giving $C\in\RR^{m\times n^2}$ and
$d\in\RR^{m}$ for the affine constraints. The matrix $A$ consists of rows
from the $65{,}536\times65{,}536$ identity matrix indexed by known pixels, while
$C$ contains the corresponding DCT rows.

In implementation, neither $A$ nor $C$ is formed explicitly; we apply
their actions and adjoints implicitly to avoid large matrix operations.

When running FISTA, we set the stepsize $1/\lip  = 1$ since $A$
consists of rows of the identity and thus satisfies $\norm{A\tran A} = 1$. The parameter sequence $(t_k)_{k \in \NN}$ was chosen to be,
\begin{equation}
    t_{0} =1\;\;\text{and}\; \; (\forall k\geq1)\; t_{k+1} = \frac{1+\sqrt{1+4t_k^2}}{2}.
\end{equation}
The stopping criterion is based on the norm of the composite gradient mapping, which
in this setting, recalling \cref{eq:PUformula},  is
\begin{equation}
\label{eq:Gx}
G(x) = x - P_U\big(x - \nabla f(x)\big)
      =(\Id-C\tran C)A\tran (Ax-b)+C\tran (Cx-d).
\end{equation}
In our numerical experiment, FISTA was terminated when 
$\norm{G(x_k)}\le 10^{-12}$.

At most $47{,}192$ measurements are used in the 
reconstruction. We deliberately avoid using all uncorrupted pixels 
and the full DCT information so that the problem remains underdetermined and admits multiple optimal solutions. In 
\cref{fig:recon_final}, we display reconstructions obtained from different initializations. All three runs 
terminated when the norm of the composite gradient mapping met the prescribed tolerance; as seen below, different 
initializations yield different solutions. We also note slight background noise around the bird in 
\cref{fig:recon_ones,fig:recon_zeros}, which arises from using information from a previous frame. Finally, 
starting from a random initialization produces a noisier final image compared to the image resulting from starting at the vector of zeros or vector of ones. 
\emph{This demonstrates the practical significance of identifying the limit of the APG scheme: in 
underdetermined settings, the initialization identifies the solution closest to the initial point.}

\begin{figure}[H]
    \centering
    \subfloat[]{
        \includegraphics[width=0.3\textwidth]{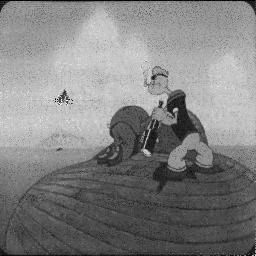}
        \label{fig:recon_ones}
    }
    \hfill
    \subfloat[]{
        \includegraphics[width=0.30\textwidth]{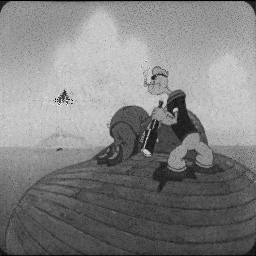}
        \label{fig:recon_zeros}
    }
    \hfill
    \subfloat[]{
        \includegraphics[width=0.3\textwidth]{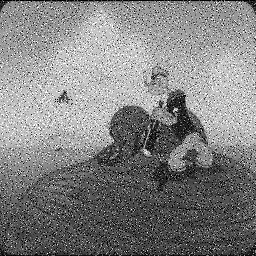}
        \label{fig:recon_rand}
    }
        \caption{\cref{fig:recon_ones} reconstruction from $x_0=(1,1, \ldots,1)$, 
        \cref{fig:recon_zeros} reconstruction from $x_0=(0,0, \ldots,0)$, and \cref{fig:recon_rand}
        reconstruction from a random $x_0\in \RR^{65{,}536}$.}
    \label{fig:recon_final}
\end{figure}

\appendix

\renewcommand\thesection{\Alph{section}}

\numberwithin{equation}{section}
\numberwithin{proposition}{section}

\crefname{appendix}{appendix}{appendices}
\Crefname{appendix}{Appendix}{Appendices}
\crefalias{section}{appendix} 

\section{Appendix A}
\label{app:0}

\begin{namedproof}{Proof of \cref{lem:key}}

\cref{lem:key:i}:
Let $x^*  \in S$.
The left inequality 
in \cref{e:ineq1} follows 
because $\mu=\min F(X)$.
Let $k \ge 1$ and observe that 
\begin{equation}
 \|z_k - x^*\|^2 =
 \|t_k y_k +(1-t_k )x_k -x^* \|^2 .
\label{eq:z-norm-eq}
\end{equation}

It follows from \cref{30.2:b},
\cref{eq:delta-def},
the convexity of $F$,
\cref{lem:29.2},
\cref{30.3a},
\cref{eq:zk-def} and \cref{eq:z-norm-eq}
that 
\begin{align*}
t_{k-1}^2 \delta_k - t_k^2 \delta_{k+1}
&\ge (t_k^2 - t_k)\delta_k - t_k^2\delta_{k+1} 
\\
&= t_k^2\!\left(\Big(1-\tfrac{1}{t_k}\Big)\delta_k - \delta_{k+1}\right) 
\\
&= t_k^2\!\left(\Big(1-\tfrac{1}{t_k}\Big)\!\big(F(x_k)-F(x^*)\big) - \big(F(x_{k+1})-F(x^*)\big)\right) 
\\
&= t_k^2\!\left(\Big(1-\tfrac{1}{t_k}\Big)F(x_k) + \tfrac{1}{t_k}F(x^*) - F(x_{k+1})\right) \\
&\ge t_k^2\!\left(F\!\left(\tfrac{1}{t_k}x^* + \Big(1-\tfrac{1}{t_k}\Big)x_k\right) - F(x_{k+1})\right)
\\
&\ge \frac{t_k^2 \lip}{2}\!\left(\Big\|\tfrac{1}{t_k}x^* + \Big(1-\tfrac{1}{t_k}\Big)x_k - x_{k+1}\Big\|^2
- \Big\|\tfrac{1}{t_k}x^* + \Big(1-\tfrac{1}{t_k}\Big)x_k - y_k\Big\|^2 \right) \\
&= \frac{\lip}{2}\!\left(\|x^* + (t_k-1)x_k - t_k x_{k+1}\|^2 - \|x^* + (t_k-1)x_k - t_k y_k\|^2\right) \\
&= \frac{\lip}{2}\!\left(\|z_{k+1} - x^*\|^2 - \|z_k - x^*\|^2\right)
.
\end{align*}

Rearranging yields
\begin{equation}
\xi_{k+1}=t_k^2 \delta_{k+1} + \frac{\lip}{2}\|z_{k+1} - x^*\|^2
\le
t_{k-1}^2 \delta_k + \frac{\lip}{2}\|z_k - x^*\|^2 =\xi_{k}.
\label{eq:master-ineq}
\end{equation}
This proves that
 \begin{equation}
\xi_{k+1}\le \xi_{k}\le \ldots \le \xi_1.
 \end{equation}
 We now turn to the right inequality in \cref{e:ineq1}. Observe that
 $\xi_1=t_0^2\big(F(x_1) - \mu\big) +\tfrac{\lip}{2} \|z_1 - x^*\|^2$.
Since $t_0=1$, $x_0=y_0$, and hence $y_1=x_1$ (by \cref{30.3b}), we have $z_1=x_1$ and thus
\begin{equation}
\xi_1=\big(F(x_1) - \mu\big) +\tfrac{\lip}{2} \|x_1 - x^*\|^2.
\end{equation}

Using \cref{lem:29.2} 
applied with $(x,y)$
replaced by $(x^*,x_0)$
and recalling $Tx_0 = Ty_0 = x_1$, we estimate
\begin{equation}
\mu - F(x_1) = F(x^*) - F(x_1)
\ge \tfrac{\lip}{2}\|x_1-x^* \|^2 - \tfrac{\lip}{2}\|x_0-x^* \|^2;
\end{equation}
equivalently,
\begin{equation}
F(x_1) - \mu
\le \frac{\lip}{2} \big(\|x_0-x^* \|^2 - \|x_1-x^* \|^2 \big).
\label{eq:309}
\end{equation}
Altogether,
we learn that 
$\xi_1\le \tfrac{\lip}{2}\norm{x_0-x^*}^2$
and \cref{e:ineq1} is verified.

\cref{lem:key:i}:
It follows from \cref{e:ineq1}
that $(\xi_k)_{k\in\NN}$
is a bounded nonincreasing sequence of real numbers 
hence it converges.

\cref{lem:key:ii}:
This follows from 
\cref{lem:key:i} and \cref{e:ineq1}
by observing that 
$(\xi_k)_{k\in\NN}$
is the sum of two nonnegative sequences, 
namely $(t_{k-1}^2 \delta_k )_{k\in\NN}$
 and $ (\frac{\lip}{2}\|z_k - x^*\|^2)_{k\in\NN}$.
 \end{namedproof}

\section{Appendix B}
\label{app:A}

\begin{namedproof}{Proof of \cref{prop:MAP-to-10}}
We proceed by induction.
Base case at $k=1$. Using \cref{eq:PVPU} we have
\begin{equation}
\label{eq:x1}
p_1=P_UP_V(x_0)=P_U(x_0)
=\Big(\tfrac{w-0+1}{2},\ \tfrac{0-w+1}{2}\Big)
=\Big(1+\tfrac{w-1}{2},\ -\,\tfrac{w-1}{2}\Big).
\end{equation}
This verifies the base case.
Suppose for some $k\ge 1$, we have 
\begin{equation}
\label{eq:xkIH}
p_k=\Big(1+\tfrac{w-1}{2^{k}},\ -\,\tfrac{w-1}{2^{k}}\Big).
\end{equation}
By \cref{eq:PVPU} and \cref{eq:xkIH} we have 
\begin{equation}
\label{eq:PVxk}
p_{k+1}=P_UP_V(p_k)=P_U\Big(1+\tfrac{w-1}{2^{k}},\ 0\Big)
=\Big(1+\tfrac{w-1}{2^{k+1}},\ -\,\tfrac{w-1}{2^{k+1}}\Big).
\end{equation}
This prove that  $(p_k)_{k\ge 1}=\big(\big(1+\tfrac{w-1}{2^{k}},\ -\,\tfrac{w-1}{2^{k}}\big)\big)_{k\ge 1}$.
The limit follows from the fact that $\tfrac{w-1}{2^{k}}\to 0$ as $k\to\infty$.
The proof is complete.
\end{namedproof}

\section{Appendix C}
\label{app:B}

\begin{namedproof}{Proof of \cref{lem:aux}}
\cref{lem:aux:i}:
Since $a_M=a\in\,]0,1[$ and $\frac{k-1}{k+2}\in\,]0,1[$ for every $k\ge M\ \ge 2$, 
we get $a_{k+1}=\frac{k-1}{k+2}a_k\in\,]0,a_k[$.
The conclusion follows by simple induction on $k\ge M$.

\cref{lem:aux:ii}:
Indeed,
\[
a_k
= a_M\prod_{j=M}^{k-1}\frac{j-1}{j+2}
= a\,\frac{\prod_{j=M}^{k-1}(j-1)}{\prod_{j=M}^{k-1}(j+2)}
= a\,\frac{(M-1)M(M+1)}{(k-1)k(k+1)}.
\]

\cref{lem:aux:iii}:
This is a direct consequence of \cref{lem:aux:ii}.

\cref{lem:aux:iv}:
Indeed, we have 
\begin{subequations}
\begin{align}
\sum_{j=M}^{k} a_{j+1}
&= a\,(M-1)M(M+1)\sum_{j=M}^{k}\frac{1}{j(j+1)(j+2)}
\\
&= \frac{a}{2} (M-1)M(M+1)
\sum_{j=M}^{k}
\Bigl(\frac{1}{j(j+1)}-\frac{1}{(j+1)(j+2)}\Bigr)
\\
&=
\frac{a}{2} (M-1)M(M+1)
\Bigl(\frac{1}{M(M+1)}-\frac{1}{(k+1)(k+2)}\Bigr)
\\
&=
\frac{a\,(M-1)}{2}\,\Bigl(1-\frac{M(M+1)}{(k+1)(k+2)}\Bigr).
\end{align}
\end{subequations}
The claim about the infinite sum follows by taking
the limit as 
$k\to\infty$.

\cref{lem:aux:v}:
If $k=M$ the conclusion follows from  \cref{lem:aux:iv}.
Now suppose $k\ge M+1$. Applying \cref{lem:aux:iv} with $k$ replaced by $k-1$, we have
\begin{subequations}
\begin{align}
   \sum_{j=k}^{\infty} a_{j+1}
&=\sum_{j=M}^{\infty} a_{j+1}-\sum_{j=M}^{k-1} a_{j+1}
\\
&=\frac{a\,(M-1)}{2}-\frac{a\,(M-1)}{2}\,\Bigl(1-\frac{M(M+1)}{k(k+1)}\Bigr)
\\
&=\frac{a\,(M-1)}{2}\cdot\frac{M(M+1)}{k(k+1)}. 
\end{align}
\end{subequations}
\cref{lem:aux:vi}:
It follows from \cref{lem:aux:ii} 
and \cref{lem:aux:v} applied with $k$ replaced by $k+1$ that  
\begin{equation}
\frac{a_{k+1}}{\sum_{j=k+1}^{\infty} a_{j+1}}
=\frac{a\,\cdot\frac{(M-1)M(M+1)}{k(k+1)(k+2)}}{a\,\cdot\frac{(M-1)M(M+1)}{2\,(k+1)(k+2)}}
=\frac{2}{k}\ \le\ 1.
\end{equation}
The proof is complete.
\end{namedproof}

\section{Appendix D}
\label{app:C}

\begin{namedproof}{Proof of \cref{prop:FISTAcone}}
By definition, $(x_k)_{k\ge 1}$ lies in $U$. Hence, $(y_k)_{k\ge 2}=((1+\alpha_{k-1})x_k-\alpha_{k-1}x_{k-1})_{k\ge 2}$ lies in $U$.
Write 
\begin{equation}
\label{eq:y-form-M}
    y_k=\bigl(1-u_k-\alpha_{k-1}d_k,\ u_k+\alpha_{k-1}d_k\bigr)\quad (k\ge 2).
\end{equation}
We prove the following. 
\textsc{Claim:}
$(\forall k\ge M)$ the following hold.
\begin{enumerate}[(A)]
\item 
\label{prop:FISTAcone:A}
$y_k\in U\cap V$.
\item 
\label{prop:FISTAcone:B}
$u_{k+1}=u_k+\alpha_{k-1}d_k$, and $d_{k+1}=\alpha_{k-1}d_k=\dfrac{k-1}{k+2}\,d_k\ (\ge 0)$.
\item 
\label{prop:FISTAcone:C}
$u_k\ \ge\ u_M,\qquad
u_k+d_k\ \le\ u_M+\tfrac{M-1}{2}\,d_M$.
\end{enumerate}
Consequently, 
$(\forall k\ge M)$, we have 
\begin{equation}
d_k=d_M\,\frac{(M-1)M(M+1)}{(k-1)k(k+1)},\qquad
\lim_{k\to\infty}u_k=u_M+\frac{M-1}{2}\,d_M=u^*\in\left[0,1\right].
\end{equation}

We proceed by strong induction on $k\ge M$.

Base case at $k=M$.
By assumption, $y_M\in U\cap V$, so $x_{M+1}=P_UP_V(y_M)=y_M\in U\cap V$.
Thus \cref{eq:y-form-M} applied with $k$
replaced by $M$ yields
\[
u_{M+1}=u_M+\alpha_{M-1} d_M,\qquad d_{M+1}=u_{M+1}-u_M=\alpha_{M-1} d_M\ge 0.
\]
Using $d_M\ge 0$ and the assumption $u_M+\tfrac{M-1}{2}\,d_M\le 1$, we have
$u_M\le 
u_M+d_M\le
u_M+\tfrac{M-1}{2}\,d_M\le 1$, verifying the base case.

Let $(a_k)_{k\ge M}$ be defined as in \cref{lem:aux} with  $a=a_M:=d_M$.
Let $k\ge M$.
Suppose that $(\forall m\in\{M,\dots,k\})$ we have 
\[
y_m\in U\cap V,\quad
d_m=a_m,\quad
u_m\ge u_M,\quad
u_m+d_m\le u_M+\tfrac{M-1}{2}d_M\ \le 1.
\]
We prove \cref{prop:FISTAcone:A}--\cref{prop:FISTAcone:C} hold for $k+1$.
Since $y_k\in U\cap V$, we have 
\begin{equation}
(1-u_{k+1},u_{k+1})=x_{k+1}=P_UP_V(y_k)=y_k\in U\cap V,
\end{equation}
and thus, recalling \cref{eq:y-form-M}, 
\begin{equation}
\label{eq:indsB-M}
u_{k+1}=u_k+\alpha_{k-1}d_k,\qquad
d_{k+1}=u_{k+1}-u_k=\alpha_{k-1}d_k=\frac{k-1}{k+2}\,d_k\ \ge 0.
\end{equation}
This verifies \cref{prop:FISTAcone:B}
at $k+1$.
The inductive hypothesis and \cref{eq:indsB-M} yield $d_k=a_k$, and by the definition of $a_{k+1}$ we get $d_{k+1}=a_{k+1}$.
Observe that, by anti-telescoping and using 
the inductive hypothesis together with \cref{eq:indsB-M}, we have
\begin{equation}
\label{eq:indsB:i-M}
u_{k+1}-u_M=\sum_{r=M}^{k} d_{r+1}\ge 0.
\end{equation}
On the one hand, \cref{eq:indsB:i-M} yields $u_{k+1}\ge u_M$.
On the other hand, \cref{eq:indsB:i-M},
\cref{lem:aux}\cref{lem:aux:v}\&\cref{lem:aux:iv}
imply 
\begin{subequations}
\begin{align}
u_{k+1}+d_{k+1}
&=u_M+\sum_{r=M}^{k} d_{r+1}+d_{k+1}
\\
&=u_M+\sum_{r=M}^{k} a_{r+1}+a_{k+1}
\\
&\le u_M+\sum_{r=M}^{k} a_{r+1}
+\sum_{r=k+1}^{\infty} a_{r+1}
\\
&=u_M+\sum_{r=M}^{\infty} a_{r+1}
= u_M+\frac{M-1}{2}\,a_M
\\
&
= u_M+\frac{M-1}{2}\,d_M\ \le\ 1.
\end{align}
\end{subequations}
 This verifies \cref{prop:FISTAcone:C}
 at $k+1$.
Finally, using $y_{k+1}=(1-u_{k+1}-\alpha_k d_{k+1},\,u_{k+1}+\alpha_k d_{k+1})$ and $\alpha_k\in\,]0,1[$, we obtain
\[
(y_{k+1})_2
=u_{k+1}+\alpha_k d_{k+1}
\in[\,u_{k+1},\,u_{k+1}+d_{k+1}\,]\subseteq\Big[\,u_M,\,u_M+\tfrac{M-1}{2}d_M\,\Big]\subseteq[0,1],
\]
so $y_{k+1}\in U\cap V$, which is \cref{prop:FISTAcone:A} at $k+1$.
Therefore, 
$(\forall k\ge M)$
\cref{prop:FISTAcone:A}--\cref{prop:FISTAcone:C} hold. 

The “Consequently” part follows from applying \cref{lem:aux}\cref{lem:aux:ii}–\cref{lem:aux:iv} with $(a_k)_{k\ge M}$ replaced by $(d_k)_{k\ge M}$,
in view of \cref{eq:indsB:i-M}, yielding
\begin{equation}
d_k
= d_M\,\frac{(M-1)M(M+1)}{(k-1)k(k+1)},\qquad
\lim_{k\to \infty} u_k
= u_M+\frac{M-1}{2}\,d_M=u^*,
\end{equation}
and hence $\lim_{k\to\infty}x_k=(1-u^*,u^*)\in U\cap V$.
This completes the proof.
\end{namedproof}

\section{Appendix E}
\begin{namedproof}{Proof of \cref{prop:example-w2}}
\label{app:D}
The claim 
$\lim_{k\to \infty} p_k=(1,0)$
follows from
\cref{prop:MAP-to-10} applied with $w=5$.
We now turn to $\lim_{k\to \infty} x_k$.
We claim that $M=4$ satisfies the assumptions of \cref{prop:FISTAcone}. Indeed, 
\begin{equation}
x_1=(3,-2),\quad
x_2=(2,-1),\quad
x_3=\tfrac{1}{8}(11,-3),\quad
x_4=\tfrac{1}{16}(17,-1),
\end{equation}
 and so
\begin{equation}
d_4=u_4-u_3=-\tfrac{1}{16}-\Bigl(-\tfrac{3}{8}\Bigr)=\tfrac{5}{16}>0.
\end{equation}
With $\alpha_3=\tfrac{1}{2}\in\,]0,1[$ and $x_4-x_3=\bigl(-\tfrac{5}{16},\tfrac{5}{16}\bigr)$, we obtain
\begin{equation}
y_4=x_4+\alpha_3(x_4-x_3)
=\Bigl(\tfrac{29}{32},\tfrac{3}{32}\Bigr)\in U\cap V,
\end{equation}
establishing $y_4\in U\cap V$.
It remains to verify that $u_M+\tfrac{M-1}{2}\,d_M\le 1$ at $M=4$. Indeed,
\begin{equation}
u_4+\tfrac{4-1}{2}\,d_4
= -\tfrac{1}{16} + \tfrac{3}{2}\cdot\tfrac{5}{16}
= \tfrac{13}{32}\ \le\ 1.
\end{equation}
Thus all assumptions of \Cref{prop:FISTAcone} hold with $M=4$.
Consequently, by \Cref{prop:FISTAcone}, for every $k\ge 4$,
\begin{equation}
d_k=\frac{75}{4}\cdot\frac{1}{(k-1)k(k+1)},\qquad
u_k=\frac{13}{32}-\frac{75}{8\,k(k+1)}\to \frac{13}{32},
\end{equation}
and
\begin{equation}
x_k=(1-u_k,u_k)
\ \longrightarrow\ \Bigl(\tfrac{19}{32},\tfrac{13}{32}\Bigr)\in U\cap V.
\end{equation}
The proof is complete.
\end{namedproof}
                                        
\end{document}